\documentclass{article}
\usepackage{amsthm}
\usepackage{amssymb}
\usepackage{amsfonts}
\usepackage{amsmath}
\usepackage{xcolor}
\usepackage{tikz}
\usepackage{hyperref}

\usepackage{geometry}
\begin{document}

\newtheorem{theorem}{Theorem}[section]
\newtheorem*{theorem*}{Theorem}
\newtheorem*{BGtheorem}{Busch--Gudder Theorem}
\newtheorem*{Mtheorem}{Moln\'ar's Theorem}
\newtheorem*{Ltheorem}{Ludwig's Theorem}
\newtheorem*{Wtheorem}{Wigner's Theorem}
\newtheorem{defi}{Definition}[section]
\newtheorem{corollary}[theorem]{Corollary}
\newtheorem{lemma}[theorem]{Lemma}
\newtheorem{proposition}[theorem]{Proposition}
\newtheorem{conjecture}[theorem]{Conjecture}
\newtheorem{step}[theorem]{Step}
\newtheorem{example}[theorem]{Example}
\newtheorem{remark}[theorem]{Remark}
\newtheorem{question}[theorem]{Question}

\newcommand{\tr}{{\rm tr}}
\newcommand{\supp}{{\rm supp}}
\newcommand{\linspan}{{\rm span}}
\newcommand{\rank}{{\rm rank}}
\newcommand{\diag}{{\rm Diag}}
\newcommand{\Image}{{\rm Im\,}}
\newcommand{\Ker}{{\rm Ker}}
\newcommand\Bdd{{\cal B}}
\newcommand\Effect{{\cal E}}
\newcommand\Proj{{\cal P}}
\newcommand\Sca{\mathcal{SC}}
\newcommand\Borel{{\cal B}}
\newcommand\calC{{\cal C}}
\newcommand\calT{{\cal T}}
\newcommand\FiniteRank{{\cal F}}
\newcommand\R{\mathbb{R}}
\newcommand\N{\mathbb{N}}
\newcommand\C{\mathbb{C}}

\newcommand\E{\ell}
\newcommand\calM{{\cal M}}
\newcommand\pc{\mathfrak{c}}

\title{Coexistency on Hilbert space effect algebras and a characterisation of its symmetry transformations
\thanks{The first author was supported by the Leverhulme Trust Early Career Fellowship, ECF-2018-125. He was also partly supported by the Hungarian National Research, Development and Innovation Office -- NKFIH (K115383).}
\thanks{The second author was supported by grants N1-0061, J1-8133, and P1-0288 from ARRS, Slovenia.}}
\author{Gy\" orgy P\' al Geh\' er \footnote{Department of Mathematics and Statistics, University of Reading, Whiteknights, P.O.~Box 220, Reading RG6 6AX, United Kingdom, G.P.Geher@reading.ac.uk or gehergyuri@gmail.com}, \quad
Peter \v Semrl\footnote{Faculty of Mathematics and Physics, University of Ljubljana,
        Jadranska 19, SI-1000 Ljubljana, Slovenia; Institute of Mathematics, Physics, and Mechanics, Jadranska 19, SI-1000 Ljubljana, Slovenia, peter.semrl@fmf.uni-lj.si}
        }

\date{}

\maketitle

\begin{abstract}
	The Hilbert space effect algebra is a fundamental mathematical structure which is used to describe unsharp quantum measurements in Ludwig's formulation of quantum mechanics.
	Each effect represents a quantum (fuzzy) event.
	The relation of coexistence plays an important role in this theory, as it expresses when two quantum events can be measured together by applying a suitable apparatus.
	This paper's first goal is to answer a very natural question about this relation, namely, when two effects are coexistent with exactly the same effects?
	The other main aim is to describe all automorphisms of the effect algebra with respect to the relation of coexistence.
	In particular, we will see that they can differ quite a lot from usual standard automorphisms, which appear for instance in Ludwig's theorem.
	As a byproduct of our methods we also strengthen a theorem of Moln\'ar.
\end{abstract}
\maketitle



%-------------------------------------------------------------------------------------------------------

%\tableofcontents

\section{Introduction}

\subsection{On the classical mathematical formulation of quantum mechanics}
Throughout this paper $H$ will denote a complex, not necessarily separable, Hilbert space with dimension at least 2.
In the classical mathematical formulation of quantum mechanics such a space is used to describe experiments at the atomic scale. 
For instance, the famous Stern--Gerlach experiment (which was one of the firsts showing the reality of the quantum spin) can be described using the two-dimensional Hilbert space $\C^2$. 
In the classical formulation of quantum mechanics, the space of all rank-one projections $\Proj_1(H)$ plays an important role, as its elements represent so-called quantum pure-states (in particular in the Stern-Gerlach experiment they represent the quantum spin).
The so-called transition probability between two pure states $P, Q \in \Proj_1(H)$ is the number $\tr PQ$, where $\tr$ denotes the trace.
For the physical meaning of this quantity we refer the interested reader to e.g.~\cite{Uhlmann}.
A very important cornerstone of the mathematical foundations of quantum mechanics is \emph{Wigner's theorem}, which states the following.

\begin{Wtheorem}
Given a bijective map $\phi\colon \Proj_1(H)\to\Proj_1(H)$ that preserves the transition probability, i.e.~$\tr \phi(P)\phi(Q) = \tr PQ$ for all $P,Q\in\Proj_1(H)$, one can always find either a unitary, or an antiunitary operator $U\colon H\to H$ that implements $\phi$, i.e.~we have $\phi(P) = UPU^*$ for all $P\in\Proj_1(H)$.
\end{Wtheorem}

For an elementary proof see \cite{GeherWigner}. 
As explained thoroughly by Simon in \cite{Simon}, this theorem plays a crucial role (together with Stone's theorem and some representation theory) in obtaining the general time-dependent Schr\"odinger equation that describes quantum systems evolving in time (and which is usually written in the form $i \hslash \tfrac{d}{dt} |\Psi(t)\rangle = \hat{H} |\Psi(t)\rangle$, where $\hslash$ is the reduced Planck constant, $\hat{H}$ is the Hamiltonian operator, and $|\Psi(t)\rangle$ is the unit vector that describes the system at time $t$).

One of the main objectives of quantum mechanics is the study of measurement.
In the classical formulation an observable (such as the position/momentum of a particle, or a component of a particle's spin) is represented by a self-adjoint operator. Equivalently, we could say that an observable is represented by a projection-valued measure $E\colon \Borel_\R \to \Proj(H)$ (i.e.~the spectral measure of the representing self-adjoint operator), where $\Borel_\R$ denotes the set of all Borel sets in $\R$ and $\Proj(H)$ the space of all projections (also called sharp effects) acting on $H$.
If $\Delta$ is a Borel set, then the quantum event that we get a value in $\Delta$ corresponds to the projection $E(\Delta)$.
However, this mathematical formulation of observables implicitly assumes that measurements are perfectly accurate, which is far from being the case in real life.
This was the crucial thought which led Ludwig to give an alternative axiomatic formulation of quantum mechanics which was introduced in his famous books \cite{LudI} and \cite{LudII}.

\subsection{On Ludwig's mathematical formulation of quantum mechanics}
This paper is related to \emph{Ludwig's formulation of quantum mechanics}, more precisely, we shall examine one of the theory's most important relations, called \emph{coexistence} (see the definition later).
The main difference compared to the classical formulation is that (due to the fact that no perfectly accurate measurement is possible in practice) quantum events are not sharp, but fuzzy.
Therefore, according to Ludwig, a quantum event is not necessarily a projection, but rather a self-adjoint operator whose spectrum lies in $[0,1]$.
Such an operator is called an \emph{effect}, and the set of all such operators is called the \emph{Hilbert space effect algebra}, or simply the \emph{effect algebra}, which will be denoted by $\Effect(H)$.
Clearly, we have $\Proj(H)\subset \Effect(H)$.
A fuzzy or unsharp quantum observable corresponds to an effect-valued measure on $\Borel_\R$, which is often called a normalised positive operator-valued measure, see  e.g.~\cite{Gudder} for more details on this. 
We point out that the role of effects and positive operator-valued measures was already emphasised in the earlier book \cite{Davies} of Davies. For some of the subsequent contributions to the theory we refer the reader to the work of Kraus \cite{Kraus} and the recent book of Busch--Lahti--Pellonp\"a\"a--Ylinen \cite{BuLa}.

Let us point out that, contradicting to its name, $\Effect(H)$ is obviously not an actual algebra.
There are a number of operations and relations on the effect algebra that are relevant in mathematical physics.
First of all, the usual \emph{partial order} $\leq$, defined by $A\leq B$ if and only if $\langle Ax, x \rangle \leq \langle B x , x \rangle$ for all $x \in H$, expresses that the occurrence of the quantum event $A$ implies the occurrence of $B$.
We emphasise that $(\Effect(H),\leq)$ is not a lattice, because usually there is no largest effect $C$ whose occurrence implies both $A$ and $B$ (see \cite{Ando, Moreland-Gudder, Tamas} for more details on this).
Note that, as can be easily shown, we have $\Effect(H) = \{A \in \Bdd(H) \colon A=A^*, 0\leq A\leq I\}$, where $\Bdd(H)$ denotes the set of all bounded operators on $H$, $A^*$ the adjoint of $A$, and $I$ the identity operator. 
Hence sometimes the literature refers to $\Effect(H)$ as the operator interval $[0,I]$.

Second, the so called \emph{ortho-complementation} $\perp$ is defined by $A^\perp = I - A$, and it can be thought of as the complement event (or negation) of $A$, i.e.~$A$ occurs if and only if $A^\perp$ does not.

We are mostly interested in the relation of \emph{coexistence}.
Ludwig called two effects coexistent if they can be measured together by applying a suitable apparatus.
In the language of mathematics (see \cite[Theorem IV.1.2.4]{LudI}), this translates into the following definition: 
\begin{defi}
$A, B \in \Effect(H)$ are said to be coexistent, in notation $A \sim B$, if there are effects $E,F,G \in \Effect(H)$ such that
\begin{equation*}
	A = E + G, \quad B = F + G \quad\text{and}\quad E+F+G \in \Effect(H).
\end{equation*}
\end{defi}
We point out that in the earlier work \cite{Davies} Davies examined the simultaneous measurement of unsharp position and momentum, which is closely related to coexistence. It is apparent from the definition that coexistence is a symmetric relation. Although it is not trivial from the above definition, two sharp effects $P,Q \in \Proj(H)$ are coexistent if and only if they commute (see Section \ref{sec:2}), which corresponds to the classical formulation.
We will denote the set of all effects that are coexistent with $A\in \Effect(H)$ by 
\begin{equation*}
	A^\sim := \{ C\in \Effect(H)\colon C\sim A \}, 
\end{equation*}	
and more generally, if $\calM \subset \Effect(H)$, then $\calM^\sim := \cap\{ A^\sim \colon A\in \calM\}$.

The relation of order in $\Effect(H)$ is fairly well-understood.
However, the relation of coexistence is very poorly understood. 
In the case of qubit effects (i.e.~when $\dim H=2$) the recent papers of Busch--Schmidt \cite{BuS}, Stano--Reitzner--Heinosaari \cite{Stano} and Yu--Liu--Li--Oh \cite{Yu} provide some (rather complicated) characterisations of coexistence.
Although there are no similar results in higher dimensions, it was pointed out by Wolf--Perez-Garcia--Fernandez in \cite{Wolf} that the question of coexistence of pairs of effects can be phrased as a so-called semidefinite program, which is a manageable numerical mathematical problem.
We also mention that Heinosaari--Kiukas--Reitzner in \cite{H} generalised the qubit coexistence characterisation to pairs of effects in arbitrary dimensions that belong to the von Neumann algebra generated by two projections.

To illustrate how poorly the relation of coexistence is understood, we note that the following very natural question has not been answered before -- not even for qubit effects:
\begin{center}
\emph{What does it mean for two effects $A$ and $B$ to be coexistent with exactly the same effects?}
\end{center}
As our first main result we answer this very natural question.
Namely, we will show the following theorem, where $\FiniteRank(H)$ and $\Sca(H)$ denote the set off all finite-rank and scalar effects on $H$, respectively.
	
\begin{theorem}\label{thm:CoexSet}
	For any effects $A,B\in \Effect(H)$ the following are equivalent:
	\begin{itemize}
		\item[\textup{(i)}] $B\in\{A, A^\perp\}$ or $A,B \in \Sca(H)$,
		\item[\textup{(ii)}] $A^\sim = B^\sim$.
	\end{itemize}
	Moreover, if $H$ is separable, then the above statements are also equivalent to
	\begin{itemize}
		\item[\textup{(iii)}] $A^\sim\cap\FiniteRank(H) = B^\sim\cap\FiniteRank(H)$.
	\end{itemize}
\end{theorem}

Physically speaking, the above theorem says that the (unsharp) quantum events $A$ and $B$ can be measured together with exactly the same quantum events if and only if they are the same, or they are each other's negation, or both of them are scalar effects.

\subsection{Automorphisms of $\Effect(H)$ with respect to two relations}
Automorphisms of mathematical structures related to quantum mechanics are important to study because they provide the right tool to understand the time-evolution of certain quantum systems (see e.g.~\cite[Chapters V-VII]{LudI} or \cite{Simon}).
In case when this mathematical structure is $\Effect(H)$, we call a map $\phi\colon \Effect(H)\to \Effect(H)$ a \emph{standard automorphism} of the effect algebra if there exists a unitary or antiunitary operator $U\colon H\to H$ that (similarly to Wigner's theorem) implements $\phi$, i.e.~we have
	\begin{equation}\label{eq:stand}
		\phi (A) = UAU^\ast  \qquad (A\in \Effect(H)).
	\end{equation}
Obviously, standard automorphisms are automorphisms with respect to the relations of order:
\begin{equation}\label{eq:ord}\tag{$\leq$}
	A \le B \iff \phi (A) \le \phi (B) \qquad (A,B\in \Effect(H));
\end{equation}
of ortho-complementation:
\begin{equation}\label{eq:perp}\tag{$\perp$}
	\phi (A^\perp ) = \phi (A)^\perp  \qquad (A\in \Effect(H));
\end{equation}
and also of coexistence:
\begin{equation}\label{eq:coex}\tag{$\sim$}
	A \sim B \iff \phi (A) \sim \phi( B) \qquad (A,B\in \Effect(H)).
\end{equation}
One of the fundamental theorems in the mathematical foundations of quantum mechanics states that every ortho-order automorphism is a standard automorphism, which was first stated by Ludwig.
	
\begin{Ltheorem}[1954, Theorem V.5.23 in \cite{LudI}]
	Let $H$ be a Hilbert space with $\dim H \ge 2$. 
	Assume that $\phi \colon \Effect(H) \to \Effect(H)$ is a bijective map satisfying \eqref{eq:ord} and \eqref{eq:perp}.
	Then $\phi$ is a standard automorphism of $\Effect(H)$.
	Conversely, every standard automorphism satisfies \eqref{eq:ord} and \eqref{eq:perp}.
\end{Ltheorem}

We note that Ludwig's proof was incomplete and that he formulated his theorem under the additional assumption that $\dim H \ge 3$. The reader can find a rigorous proof of this version for instance in \cite{CDLL}.
Let us also point out that the two-dimensional case of Ludwig's theorem was only proved in 2001 in \cite{MoP}.

It is very natural to ask whether the conclusion of Ludwig's theorem remains true, if one replaces either \eqref{eq:ord} by \eqref{eq:coex}, or \eqref{eq:perp} by \eqref{eq:coex}.
Note that in light of Theorem \ref{thm:CoexSet}, in the former case the condition \eqref{eq:perp} becomes almost redundant, except on $\Sca(H)$.
However, as scalar effects are exactly those that are coexistent with every effect (see Section \ref{sec:2}), this problem basically reduces to the characterisation of automorphisms with respect to coexistence only -- which we shall consider later on.

In 2001, Moln\' ar answered the other question affirmatively under the assumption that $\dim H \geq 3$. 

\begin{Mtheorem}[2001, Theorem 1 in \cite{Mol0}]%\label{thm:Molnar}
	Let $H$ be a Hilbert space with $\dim H \ge 3$. Assume that $\phi \colon \Effect(H) \to \Effect(H)$ is a bijective map satisfying	\eqref{eq:ord} and \eqref{eq:coex}.
	Then $\phi$ is a standard automorphism of $\Effect(H)$.
	Conversely, every standard automorphism satisfies \eqref{eq:ord} and \eqref{eq:coex}.
\end{Mtheorem}

In this paper we shall prove the two-dimensional version of Moln\'ar's theorem.

\begin{theorem}\label{thm:2dM}
	Assume that $\phi \colon \Effect(\C^2) \to \Effect(\C^2)$ is a bijective map satisfying	\eqref{eq:ord} and \eqref{eq:coex}.
	Then $\phi$ is a standard automorphism of $\Effect(\C^2)$.
	Conversely, every standard automorphism satisfies \eqref{eq:ord} and \eqref{eq:coex}.
\end{theorem}

Note that Moln\'ar used the fundamental theorem of projective geometry to prove the aforementioned result, therefore his proof indeed works only if $\dim H \ge 3$. 
Here, as an application of Theorem \ref{thm:CoexSet}, we shall give an alternative proof of Moln\'ar's theorem that does not use this dimensionality constraint, hence fill this dimensionality gap in.
More precisely, we will reduce Moln\'ar's theorem and Theorem \ref{thm:2dM} to Ludwig's theorem (see the end of Section \ref{sec:2}).

\subsection{Automorphisms of $\Effect(H)$ with respect to only one relation}

It is certainly a much more difficult problem to describe the general form of automorphisms with respect to only one relation.
Of course, here we mean either order preserving \eqref{eq:ord}, or coexistence preserving \eqref{eq:coex} maps, as it is easy (and not at all interesting) to describe bijective transformations that satisfy \eqref{eq:perp}. 
It has been known for quite some time that automorphisms with respect to the order relation on $\Effect(H)$ may differ a lot from standard automorphisms, although they are at least always continuous with respect to the operator norm.
We do not state the related result here, but only mention that the answer finally has been given by the second author in \cite[Corollary 1.2]{Se0} (see also \cite{SeActa}).

The other main purpose of this paper is to give the characterisation of all automorphisms of $\Effect(H)$ with respect to the relation of coexistence. 
As can be seen from our result below, these maps can also differ a lot from standard automorphisms, moreover, unlike in the case of \eqref{eq:ord} they are in general not even continuous.
		
\begin{theorem}\label{thm:main}
	Let $H$ be a Hilbert space with $\dim H \geq 2$, and $\phi \colon \Effect(H) \to \Effect(H)$ be a bijective map that satisfies \eqref{eq:coex}.
	Then there exists a unitary or antiunitary operator $U\colon H\to H$ and a bijective map $g\colon [0,1] \to [0,1]$ such that we have
	\begin{equation}\label{eq:nonScaPhi}
		\{\phi(A), \phi(A^\perp)\} = \{UAU^\ast, UA^\perp U^\ast\} \qquad (A\in \Effect(H)\setminus \Sca(H))
	\end{equation}
	and
	\begin{equation}\label{eq:ScaPhi}
		\phi (tI) = g(t)I \qquad (t\in [0,1]).
	\end{equation}
	Conversely, every map of the above form preserves coexistence in both directions.
\end{theorem}

Observe that in the above theorem if we assume that our automorphism is continuous with respect to the operator norm, then up to unitary-antiunitary equivalence we obtain that $\phi$ is either the identity map, or the ortho-complementation: $A\mapsto A^\perp$.
Also note that the converse statement of the theorem follows easily by Theorem \ref{thm:CoexSet}.
As we mentioned earlier, the description of all automorphisms with respect to \eqref{eq:coex} and \eqref{eq:perp} now follows easily, namely, we get the same conclusion as in the above theorem, except that now $g$ further satisfies $g(1-t) = 1-g(t)$ for all $0\leq t\leq 1$.

\subsection{Quantum mechanical interpretation of automorphisms of $\Effect(H)$}

In order to explain the above automorphism theorems' physical interpretation, let us go back first to Wigner's theorem.
Assume there are two physicists who analyse the same quantum mechanical system using the same Hilbert space $H$, but possibly they might associate different rank-one projections to the same quantum (pure) state.
However, we know that they always agree on the transition probabilities.
Then according to Wigner's theorem, there must be either a unitary, or an antiunitary operator with which we can transform from one analysis into the other (like a ''coordinate transformation'').

For the interpretation of Ludwig's theorem, let us say there are two physicists who analyse the same quantum fuzzy measurement, but they might associate different effects to the same quantum fuzzy event. 
If we at least know that both of them agree on which pairs of effects are ortho-complemented, and which effect is larger than the other (i.e.~implies the occurrence of the other), then by Ludwig's theorem there must exist either a unitary, or an antiunitary operator that gives us the way to transform from one analysis into the other.

As for the interpretation of our Theorem \ref{thm:main}, if we only know that our physicists agree on which pairs of effects are coexistent (i.e.~which pairs of quantum events can be measured together), then there is a map $\phi$ satisfying \eqref{eq:nonScaPhi} and \eqref{eq:ScaPhi} that transforms the first physicist's analysis into the other's.

\subsection{The outline of the paper}
In the next section we will prove our first main result, Theorem \ref{thm:CoexSet}, and as an application, we prove Moln\'ar's theorem in an alternative way that works for qubit effects as well.
This will be followed by Section 3 where we prove our other main result, Theorem \ref{thm:main}, in the case when $\dim H = 2$.
Then in Section 4, using the two-dimensional case, we shall prove the general version of our result.
Let us point out once more that, unless otherwise stated, $H$ is not assumed to be separable.
We will close our paper with some discussion on the qubit case and some open problems in Sections 5--6.

%-------------------------------------------------------------------------------------------------------

\section{Proofs of Theorems \ref{thm:CoexSet}, \ref{thm:2dM}, and Moln\'ar's theorem} \label{sec:2}

We start with some definitions.
The symbol $\Proj(H)$ will stand for the set of all projections (idempotent and self-adjoint operators) on $H$, and $\Proj_1(H)$ will denote the set of all rank-one projections.
The commutant of an effect $A$ intersected with $\Effect(H)$ will be denoted by 
\begin{equation*}
	A^c := \{ C\in \Effect(H) \colon CA=AC\},
\end{equation*}
and more generally, for a subset $\calM\subset \Effect(H)$ we will use the notation $\calM^c := \cap\{ A^c \colon A\in\calM \}$.
Also, we set $A^{cc} := (A^c)^c$ and $\calM^{cc} := (\calM^c)^c$.

We continue with three known lemmas on the structure of coexistent pairs of effects that can all be found in \cite{Se2}. 
The first two have been proved earlier, see \cite{BuS, Mol}.

\begin{lemma}\label{lem:properties}
	For any $A\in \Effect(H)$ and $P \in \Proj(H)$ the following statements hold:
	\begin{itemize}
		\item[\textup{(a)}] $A^\sim = \Effect(H)$ if and only if $A \in \Sca(H)$,
		\item[\textup{(b)}] $P^\sim = P^c$,
		\item[\textup{(c)}] $A^c \subseteq A^\sim$.
	\end{itemize}
\end{lemma}

\begin{lemma}\label{lem:dirsum}
	Let $A,B \in \Effect(H)$ so that their matrices are diagonal with respect to some orthogonal decomposition $H = \oplus_{i\in \mathcal{I}} H_i$, i.e.~$A = \oplus_{i\in\mathcal{I}} A_i$ and $B = \oplus_{i\in\mathcal{I}} B_i \in \Effect(\oplus_{i\in\mathcal{I}} H_i)$.
	Then $A\sim B$ if and only if $A_i\sim B_i$ for all $i\in\mathcal{I}$.
\end{lemma}

\begin{lemma}\label{lem:MANIA}
	Let $A,B \in \Effect(H)$. Then the following are equivalent:
	\begin{itemize}
		\item[\textup{(i)}] $A \sim B$,
		\item[\textup{(ii)}] there exist effects $M,N \in \Effect(H)$ such that $M \le A$, $N \le I-A$, and $M+N = B$.
	\end{itemize}
\end{lemma}

We continue with a corollary of Lemma \ref{lem:properties}.

\begin{corollary}\label{cor:ref}
	For any effect $A$ and projection $P\in A^\sim$ we have $P\in A^c$. In particular, we have $A^\sim\cap\Proj(H) = A^c\cap\Proj(H)$.
\end{corollary}

\begin{proof}
Since coexistence is a symmetric relation, we obtain $A\in P^\sim$, which implies $AP=PA$.
\end{proof}

The next four statements are easy consequences of Lemma \ref{lem:MANIA}, we only prove two of them.

\begin{corollary}\label{cor:PerpCoex}
	For any effect $A$ we have $A^\sim = \left(A^\perp\right)^\sim$.
\end{corollary}

\begin{corollary}\label{cor:0-neighbourhood}
	Let $A\in \Effect(H)$ such that either $0\notin\sigma(A)$, or $1\notin\sigma(A)$. 
	Then there exists an $\varepsilon>0$ such that $\{C\in \Effect(H)\colon C\leq \varepsilon I\} \subseteq A^\sim$.
\end{corollary}

We recall the definition of the \emph{strength function} of $A\in \Effect(H)$:
\begin{equation*}
	\Lambda(A,P) = \max\{ \lambda \geq 0 \colon \lambda P \leq A \} \qquad (P\in \Proj_1(H)),
\end{equation*}
see \cite{BuGu} for more details and properties.

\begin{corollary}\label{cor:RankOneStrengthFunct}
	Assume that $A\in \Effect(H)$, $0 < t \leq 1$, and $P\in\Proj_1(H)$.
	Then the following conditions are equivalent:
	\begin{itemize}
		\item[\textup{(i)}] $A\sim tP$;
		\item[\textup{(ii)}] \begin{equation}\label{eq:Rank1StrengthFunc}
						t \leq \Lambda(A,P) + \Lambda(A^\perp,P).
					\end{equation}
	\end{itemize}
\end{corollary}

\begin{proof}
	By (ii) of Lemma \ref{lem:MANIA} we have $A\sim tP$ if and only if there exist $t_1, t_2 \geq 0$ such that $t = t_1 + t_2$, $t_1 P \leq A$ and $t_2 P \leq A^\perp$, which is of course equivalent to \eqref{eq:Rank1StrengthFunc}.
\end{proof}

\begin{corollary}\label{cor:kulon}
	Let $A,B \in \Effect(H)$ such that $A^\sim \subseteq B^\sim$. 
	Assume that with respect to the orthogonal decomposition $H = H_1 \oplus H_2$ the two effects have the following block-diagonal matrix forms:
	\begin{equation*}
		A = \left[\begin{matrix}
					A_1 & 0 \\
					0 & A_2
		\end{matrix}\right] \qquad \text{and} \qquad
		B = \left[\begin{matrix}
					B_1 & 0 \\
					0 & B_2
		\end{matrix}\right] \in \Effect(H_1 \oplus H_2).
	\end{equation*}
	Then we also have
	\begin{equation}\label{eq:reszek}
		A_1^\sim \subseteq B_1^\sim \qquad \text{and} \qquad A_2^\sim \subseteq B_2^\sim.
	\end{equation}
	In particular, if $A^\sim = B^\sim$, then $A_1^\sim = B_1^\sim$ and $A_2^\sim = B_2^\sim$.
\end{corollary}

\begin{proof}
	Let $P_1$ be the orthogonal projection onto $H_1$. 
	By Lemma \ref{lem:dirsum} we observe that
	\begin{align}\label{eq:SubMxCoex}
		\Bigg\{ 
			\left[\begin{matrix}
				C & 0 \\
				0 & D
			\end{matrix}\right] \in \Effect\left(H_1\oplus H_2 \right) 
			&\colon C \sim A_1, 
				D \sim A_2
		 \Bigg\} = P_1^c \cap A^\sim \nonumber \\
		&\subseteq P_1^c \cap B^\sim \nonumber = \Bigg\{ 
			\left[\begin{matrix}
				E & 0 \\
				0 & F
			\end{matrix}\right] \in \Effect\left(H_1\oplus H_2 \right) 
			\colon E \sim B_1, 
				F \sim B_2
		 \Bigg\},
	\end{align}
	which immediately implies \eqref{eq:reszek}.
\end{proof}

Next, we recall the Busch--Gudder theorem about the explicit form of the strength function, which we shall use frequently here.
We also adopt their notation, so whenever it is important to emphasise that the range of a rank-one projection $P$ is $\C\cdot x$ with some $x\in H$ such that $\|x\|=1$, we write $P_x$ instead. Furthermore, the symbol $A^{-1/2}$ denotes the algebraic inverse of the bijective restriction $A^{1/2}|_{(\Image A)^-}\colon (\Image A)^-\to \Image (A^{1/2})$, where $\cdot^-$ stands for the closure of a set. In particular, for all $x \in \Image(A^{1/2})$ the vector $A^{-1/2} x$ is the unique element in $(\Image A)^-$ which $A^{1/2}$ maps to $x$.

\begin{BGtheorem}[1999, Theorem 4 in \cite{BuGu}]%\label{thm:BuGu}
	For every effect $A \in \Effect(H)$ and unit vector $x\in H$ we have
	\begin{equation}
		\Lambda(A,P_x) = \left\{ \begin{matrix}
			\|A^{-1/2} x \|^{-2}, & \text{if} \; x \in \Image(A^{1/2}), \\
			0, & \text{otherwise.}
		\end{matrix} \right.
	\end{equation}
\end{BGtheorem}

We proceed with proving some new results which will be crucial in the proofs of our main theorems.
The first lemma is probably well-known, but as we did not find it in the literature, we state and prove it here.
Recall that WOT and SOT stand for the weak- and strong operator topologies, respectively.

\begin{lemma}\label{lem:WOTclosed}
	For any effect $A\in \Effect(H)$, the set $A^\sim$ is convex and WOT-compact, hence it is also SOT- and norm-closed.
	Moreover, if $H$ is separable, then the subset $A^\sim \cap \FiniteRank(H)$ is SOT-dense, hence also WOT-dense, in $A^\sim$.
\end{lemma}

\begin{proof}
	Let $t\in[0,1]$ and $B_1,B_2\in A^\sim$. 
	By Lemma \ref{lem:MANIA} there are $M_1,M_2,N_1,N_2\in \Effect(H)$ such that $M_1+N_1 = B_1$, $M_2+N_2 = B_2$, $M_1\leq A, N_1\leq I-A$ and $M_2\leq A, N_2\leq I-A$.
	Hence setting $M = tM_1+(1-t)M_2\in \Effect(H)$ and $N = tN_1+(1-t)N_2\in \Effect(H)$ gives $M+N = tB_1+(1-t)B_2$ and $M\leq A, N\leq I-A$, thus $tB_1+(1-t)B_2\sim A$, so $A^\sim$ is indeed convex. 
	
	Next, we prove that $A^\sim$ is WOT-compact. 
	Clearly, $\Effect(H)$ is WOT-compact, as it is a bounded WOT-closed subset of $\Bdd(H)$ (see \cite[Proposition IX.5.5]{Conway}), therefore it is enough to show that $A^\sim$ is WOT-closed. 
	Let $\{B_\nu\}_\nu \subseteq A^\sim$ be an arbitrary net that WOT-converges to $B$, we shall show that $B\sim A$ holds. 
	For every $\nu$ we can find two effects $M_\nu$ and $N_\nu$ such that $M_\nu+N_\nu = B_\nu$, $M_\nu\leq A$ and $N_\nu \leq I-A$.
	By WOT-compactness of $\Effect(H)$, there exists a subnet $\{B_\xi\}_\xi$ such that $M_\xi \to M$ in WOT with some effect $M$.
	Again, by WOT-compactness of $\Effect(H)$, there exists a subnet $\{B_\eta\}_\eta$ of the subnet $\{B_\xi\}_\xi$ such that $N_\eta \to N$ in WOT with some effect $N$. Obviously we also have $B_\eta \to B$ and $M_\eta \to M$ in WOT.
	Therefore we have $M+N = B$ and by definition of WOT convergence we also obtain $M\leq A$, $N\leq I-A$, hence indeed $B\sim A$.
	Closedness with respect to the other topologies is straightforward.
	
	Concerning our last statement for separable spaces, first we point out that for every effect $C$ there exists a net of finite rank effects $\{C_\nu\}_\nu$ such that $C_\nu \leq C$ holds for all $\nu$ and $C_\nu \to C$ in SOT.
	Denote by $E_C$ the projection-valued spectral measure of $C$, and set $C_n = \sum_{j=0}^{n} \frac{j}{n} E_C\left(\left[\frac{j}{n},\frac{j+1}{n}\right)\right)$ for every $n\in\N$.
	Clearly, each $C_n$ has finite spectrum, satisfies $C_n\leq C$, and $\|C_n - C\| \to 0$ as $n\to\infty$.
	For each spectral projection $E_C\left(\left[\frac{j}{n},\frac{j+1}{n}\right)\right)$ we can take a sequence of finite-rank projections $\{P_k^{j,n}\}_{k=1}^\infty$ such that $P_k^{j,n}\leq E_C\left(\left[\frac{j}{n},\frac{j+1}{n}\right)\right)$ for all $k$ and $P_k^{j,n} \to E_C\left(\left[\frac{j}{n},\frac{j+1}{n}\right)\right)$ in SOT as $k\to\infty$.
	Define $C_{n,k} := \sum_{j=0}^{n} \frac{j}{n} P_k^{j,n}$.
	It is apparent that $C_{n,k} \leq C_n$ for all $n$ and $k$, and that for each $n$ we have $C_{n,k} \to C_n$ in SOT as $k\to\infty$.
	Therefore the SOT-closure of $\{C_{n,k}\colon n,k\in\mathbb{N}\}$ contains each $C_n$, hence also $C$, thus we can construct a net $\{C_\nu\}_\nu$ with the required properties.

	Now, let $B\in A^\sim$ be arbitrary, and consider two other effects $M,N\in \Effect(H)$ that satisfy the conditions of Lemma \ref{lem:MANIA} (ii). 
	Set $C:= M\oplus N \in \Effect(H\oplus H)$, and denote by $E_M$ and $E_N$ the projection-valued spectral measures of $M$ and $N$, respectively. Clearly, $E_C\left(\left[\frac{j}{n},\frac{j+1}{n}\right)\right) = E_M\left(\left[\frac{j}{n},\frac{j+1}{n}\right)\right) \oplus E_N\left(\left[\frac{j}{n},\frac{j+1}{n}\right)\right)$ for each $j$ and $n$. 
	In the above construction we can choose finite-rank projections of the form $P_k^{j,n} = Q_k^{j,n}\oplus R_k^{j,n} \in\Proj(H\oplus H)$ where $Q_k^{j,n}, R_k^{j,n} \in\Proj(H)$, $Q_k^{j,n}\leq E_M\left(\left[\frac{j}{n},\frac{j+1}{n}\right)\right)$ and $R_k^{j,n}\leq E_N\left(\left[\frac{j}{n},\frac{j+1}{n}\right)\right)$ holds for all $k,n$.
	Then each element $C_\nu$ of the convergent net is an orthogonal sum of the form $M_\nu\oplus N_\nu \in \Effect(H\oplus H))$. It is apparent that $M_\nu, N_\nu\in\FiniteRank(H)$, $M_\nu \leq M$ and $N_\nu \leq N$ for all $\nu$, and that $M_\nu \to M$, $N_\nu \to N$ holds in SOT.
	Therefore $M_\nu+N_\nu\in\FiniteRank(H)\cap A^\sim$ and $M_\nu+N_\nu$ converges to $M+N = B$ in SOT, the proof is complete.
\end{proof}

We proceed to investigate when do we have the equation $A^\sim = B^\sim$ for two effects $A$ and $B$, which will take several steps.
We will denote the set of all rank-one effects by $\FiniteRank_1(H) := \{tP \colon P\in\Proj_1(H), 0<t\leq 1\}$.

\begin{lemma}\label{lem:BlockDiagSca}
	Let $H = H_1\oplus H_2$ be an orthogonal decomposition and assume that $A, B\in \Effect(H)$ have the following matrix decompositions:
	\begin{equation}\label{eq:BlockDiagForm}
		A = \left[ \begin{matrix}
			\lambda_1 I_1 & 0 \\
			0 & \lambda_2 I_2
		\end{matrix} \right] \quad \text{and} \quad
		B = \left[ \begin{matrix}
			\mu_1 I_1 & 0 \\
			0 & \mu_2 I_2
		\end{matrix} \right] \; \in \Effect(H_1\oplus H_2)
	\end{equation}
	where $\lambda_1, \lambda_2, \mu_1, \mu_2 \in [0,1]$, and $I_1$ and $I_2$ denote the identity operators on $H_1$ and $H_2$, respectively.
	Then the following are equivalent:
	\begin{itemize}
		\item[\textup{(i)}] $\Lambda(A,P) + \Lambda(A^\perp,P) = \Lambda(B,P) + \Lambda(B^\perp,P)$ holds for all $P\in \Proj_1(H)$,
		\item[\textup{(ii)}] $A^\sim \cap \FiniteRank_1(H) = B^\sim \cap \FiniteRank_1(H)$,
		\item[\textup{(iii)}] either $\lambda_1=\lambda_2$ and $\mu_1=\mu_2$, or $\lambda_1=\mu_1$ and $\lambda_2=\mu_2$, or $\lambda_1 + \mu_1 = \lambda_2 + \mu_2 =1$.
	\end{itemize}
\end{lemma}

\begin{proof}
	The directions (iii)$\Longrightarrow$(ii)$\iff$(i) are trivial by Lemma \ref{lem:properties} (a) and Corollaries \ref{cor:PerpCoex}, \ref{cor:RankOneStrengthFunct}, so we shall only consider the direction (i)$\Longrightarrow$(iii).
	First, a straightforward calculation using the Busch--Gudder theorem gives the following for every $x_1\in H_1, x_2\in H_2, \|x_1\| = \|x_2\| = 1$ and $0\leq\alpha\leq\tfrac{\pi}{2}$:
	\begin{equation}\label{eq:DiagStrength}
		\Lambda \left( A, P_{\cos\alpha x_1 + \sin\alpha x_2} \right)
		= \frac{1}{ \left(\tfrac{1}{\lambda_1}\right) \cdot \cos^2\alpha + \left(\tfrac{1}{\lambda_2}\right) \cdot \sin^2\alpha },
	\end{equation}
	where we use the interpretations $\tfrac{1}{0} = \infty$, $\tfrac{1}{\infty} = 0$, $\infty\cdot 0 = 0$, $\infty + \infty = \infty$, and $\infty + a = \infty$, $\infty\cdot a = \infty$ ($a>0$), in order to make the formula valid also for the case when $\lambda_1 = 0$ or $\lambda_2 = 0$.
	Clearly, \eqref{eq:DiagStrength} depends only on $\alpha$, but not on the specific choices of $x_1$ and $x_2$.
	We define the following two functions
	\begin{equation*}
		T_A\colon \left[0,\tfrac{\pi}{2}\right] \to [0,1], \qquad T_A(\alpha) = \Lambda \left( A, P_{\cos\alpha x_1 + \sin\alpha x_2} \right) + \Lambda \left( A^\perp, P_{\cos\alpha x_1 + \sin\alpha x_2} \right)
	\end{equation*}
	and
	\begin{equation*}
		T_B\colon \left[0,\tfrac{\pi}{2}\right] \to [0,1], \qquad T_B(\alpha) = \Lambda \left( B, P_{\cos\alpha x_1 + \sin\alpha x_2} \right) + \Lambda \left( B^\perp, P_{\cos\alpha x_1 + \sin\alpha x_2} \right),
	\end{equation*}
	which are the same by our assumptions.
	By \eqref{eq:DiagStrength}, for all $0\leq\alpha\leq\tfrac{\pi}{2}$ we have
	\begin{align}\label{eq:TA}
		T_A(\alpha) 
		& = \frac{1}{ \left(\tfrac{1}{\lambda_1}\right)\cdot \cos^2\alpha + \left(\tfrac{1}{\lambda_2}\right)\cdot \sin^2\alpha } + \frac{1}{ \left(\tfrac{1}{1-\lambda_1}\right)\cdot \cos^2\alpha + \left(\tfrac{1}{1-\lambda_2}\right)\cdot \sin^2\alpha } \nonumber \\ 
		& = \frac{1}{ \left(\tfrac{1}{\mu_1}\right)\cdot \cos^2\alpha + \left(\tfrac{1}{\mu_2}\right)\cdot \sin^2\alpha } + \frac{1}{ \left(\tfrac{1}{1-\mu_1}\right)\cdot \cos^2\alpha + \left(\tfrac{1}{1-\mu_2}\right)\cdot \sin^2\alpha }
		=T_B(\alpha).
	\end{align}
	
	Next, we observe the following implications:
	\begin{itemize}
		\item if $\lambda_1 = \lambda_2$, then $T_A(\alpha)$ is the constant 1 function,
		\item if $\lambda_1 = 0$ and $\lambda_2 = 1$, then $T_A(\alpha)$ is the characteristic function $\chi_{\{0,\pi/2\}}(\alpha)$,
		\item if $\lambda_1 = 0$ and $0 < \lambda_2 < 1$, then $T_A(\alpha)$ is continuous on $\left[0,\tfrac{\pi}{2}\right)$, but has a jump at $\tfrac{\pi}{2}$, namely $\lim_{\alpha\to\tfrac{\pi}{2}-} T_A(\alpha) = 1-\lambda_2$ and $T_A(\tfrac{\pi}{2}) = 1$,
		\item if $\lambda_1 = 1$ and $0 < \lambda_2 < 1$, then $T_A(\alpha)$ is continuous on $\left[0,\tfrac{\pi}{2}\right)$, but has a jump at $\tfrac{\pi}{2}$, namely $\lim_{\alpha\to\tfrac{\pi}{2}-} T_A(\alpha) = \lambda_2$ and $T_A(\tfrac{\pi}{2}) = 1$,
		\item if $\lambda_1,\lambda_2 \in (0,1)$, then $T_A(\alpha)$ is continuous on $\left[0,\tfrac{\pi}{2}\right]$,
		\item if $\lambda_1 \neq \lambda_2$, then we have $T_A(0) = T_A(\tfrac{\pi}{2}) = 1$ and $T_A(\alpha) < 1$ for all $0 < \alpha < \tfrac{\pi}{2}$.
	\end{itemize}
	All of the above statements are rather straightforward computations using the formula \eqref{eq:TA}, let us only show the last one here.
	Clearly, $T_A(0) = T_A(\tfrac{\pi}{2}) = 1$ is obvious.
	As for the other assertion, if $\lambda_1,\lambda_2 \in (0,1)$, then we can use the strict version of the weighted harmonic-arithmetic mean inequality:
	\begin{align*}
		\frac{1}{ \tfrac{1}{\lambda_1} \cos^2\alpha + \tfrac{1}{\lambda_2} \sin^2\alpha } &+ \frac{1}{ \tfrac{1}{1-\lambda_1} \cos^2\alpha + \tfrac{1}{1-\lambda_2} \sin^2\alpha } \\
		&< (\lambda_1 \cos^2\alpha + \lambda_2 \sin^2\alpha) + ((1-\lambda_1) \cos^2\alpha + (1-\lambda_2) \sin^2\alpha) = 1 \quad (0<\alpha<\tfrac{\pi}{2}).
	\end{align*}
	If $\lambda_1 = 0 < \lambda_2 < 1$, then we calculate in the following way:
	\begin{align*}
		\frac{1}{ \left(\tfrac{1}{0}\right) \cos^2\alpha + \tfrac{1}{\lambda_2} \sin^2\alpha } + \frac{1}{ \cos^2\alpha + \tfrac{1}{1-\lambda_2} \sin^2\alpha }
		= \frac{1}{ 1-\sin^2\alpha + \tfrac{1}{1-\lambda_2} \sin^2\alpha } < 1 \quad (0<\alpha<\tfrac{\pi}{2}).
	\end{align*}
	The remaining cases are very similar.
	
	The above observations together with Corollary \ref{cor:PerpCoex} and \eqref{eq:TA} readily imply the following:	
	\begin{itemize}
		\item $A\in \Sca(H)$ if and only if $B\in \Sca(H)$,
		\item $A\in \Proj(H)\setminus \Sca(H)$ if and only if $B\in \Proj(H)\setminus \Sca(H)$, in which case $B \in \{A, A^\perp\}$,
		\item there exists a $P\in\Proj(H)\setminus \Sca(H)$ and a $t\in(0,1)$ with $A\in \{tP, I-tP\}$ if and only if $B\in \{tP, I-tP\}$,
		\item $\lambda_1,\lambda_2 \in (0,1)$ and $\lambda_1\neq\lambda_2$ if and only if $\mu_1, \mu_2 \in (0,1)$ and $\mu_1\neq\mu_2$.
	\end{itemize}
	So what remained is to show that in the last case we further have $B \in \{A, A^\perp\}$, which is what we shall do below.
		
	Let us introduce the following functions:
	\begin{equation*}
		\calT_A\colon[0,1]\to [0,1], \qquad \calT_A(s) := T_A(\arcsin\sqrt{s}) = \frac{\lambda_1\lambda_2}{\lambda_1 s + \lambda_2 (1-s)} + \frac{(1-\lambda_1)(1-\lambda_2)}{(1-\lambda_1) s + (1-\lambda_2) (1-s)}
	\end{equation*}
	and
	\begin{equation*}
		\calT_B\colon[0,1]\to [0,1], \qquad \calT_B(s) := T_B(\arcsin\sqrt{s}) = \frac{\mu_1\mu_2}{\mu_1 s + \mu_2 (1-s)} + \frac{(1-\mu_1)(1-\mu_2)}{(1-\mu_1) s + (1-\mu_2) (1-s)}.
	\end{equation*}
	Our aim is to prove that $\calT_A(s) = \calT_B(s)$ ($s\in[0,1]$) implies either $\lambda_1=\mu_1$ and $\lambda_2=\mu_2$, or $\lambda_1 + \mu_1 = \lambda_2 + \mu_2 =1$.
	The derivative of $\calT_A$ is
	$$
		\tfrac{d \calT_A}{ds}(s) = (\lambda_1-\lambda_2)\left(\frac{-\lambda_1\lambda_2}{(\lambda_1 s + \lambda_2 (1-s))^2} + \frac{(1-\lambda_1)(1-\lambda_2)}{((1-\lambda_1) s + (1-\lambda_2) (1-s))^2}\right),
	$$
	from which we calculate
	$$
		\tfrac{d \calT_A}{ds}(0) = -\frac{(\lambda_1-\lambda_2)^2}{(1-\lambda_2)\lambda_2}
		\quad \text{and} \quad
		\tfrac{d \calT_A}{ds}(1) = \frac{(\lambda_1-\lambda_2)^2}{(1-\lambda_1)\lambda_1}.
	$$
	Therefore, if we managed to show that the function
	$$
		F\colon (0,1)^2 \to \R^{2}, \quad F(x,y) = \left( \tfrac{(x-y)^2}{(1-x)x}, \tfrac{(x-y)^2}{(1-y)y} \right)
	$$
	is injective on the set $\Delta := \{(x,y)\in\R^{2}\colon 0<y<x<1 \}$, then we are done (note that $F(x,y) = F(1-x,1-y)$).
	For this assume that with some $c,d>0$ we have
	$$
		\frac{(x-y)^2}{(1-x)x} = \frac{1}{c} \quad \text{and} \quad \frac{(x-y)^2}{(1-y)y} = \frac{1}{d},
	$$
	or equivalently,
	$$
		(1-x)x = c(x-y)^2 \quad \text{and} \quad (1-y)y = d(x-y)^2.
	$$
	If we substitute $u = \tfrac{x-y}{2}$ and $v = \tfrac{x+y}{2}$, then we get
	$$
		(u+v)^2-(u+v) = -4cu^2 \quad \text{and} \quad (v-u)^2-(v-u) = -4du^2.
	$$
	Now, considering the sum and difference of these two equations and manipulate them a bit gives
	$$
		v^2-v = -(2c+2d+1)u^2 \quad \text{and} \quad v = (d-c) u + \tfrac{1}{2}.
	$$
	From these latter equations we conclude
	$$
		x-y = 2u = \sqrt{\tfrac{1}{(d-c)^2+2c+2d+1}} \quad \text{and} \quad x+y = 2v = 2(d-c) u + 1,
	$$
	which clearly implies that $F$ is globally injective on $\Delta$, and the proof is complete.
\end{proof}

We have an interesting consequence in finite dimensions.

\begin{corollary}\label{cor:CoexSet2Dim}
	Assume that $2\leq \dim H < \infty$ and $A,B \in \Effect(H)$. Then the following are equivalent:
	\begin{itemize}
		\item[\textup{(i)}] $\Lambda(A, P) + \Lambda(A^\perp, P) = \Lambda(B, P) + \Lambda(B^\perp, P)$ for all $P\in P_1(H)$,
		\item[\textup{(ii)}] $A^\sim \cap \FiniteRank_1(H) = B^\sim \cap \FiniteRank_1(H)$,
		\item[\textup{(iii)}] either $A,B \in \Sca(H)$, or $A=B$, or $A=B^\perp$.
	\end{itemize}
\end{corollary}

\begin{proof}
	The directions (i)$\iff$(ii)$\Longleftarrow$(iii) are trivial, so we shall only prove the (ii)$\Longrightarrow$(iii) direction.
	First, let us consider the two-dimensional case.	
	As we saw in the proof of Lemma \ref{lem:BlockDiagSca}, we have $A^\sim \cap \FiniteRank_1(H) = \FiniteRank_1(H)$ if and only if $A$ is a scalar effect (see the first set of bullet points there).
	Therefore, without loss of generality we may assume that none of $A$ and $B$ are scalar effects.
	Notice that by Lemma \ref{lem:properties}, $A$ and $B$ commute with exactly the same rank-one projections, hence $A$ and $B$ possess the forms in \eqref{eq:BlockDiagForm} with some one-dimensional subspaces $H_1$ and $H_2$, and an easy application of Lemma \ref{lem:BlockDiagSca} gives (iii).
	
	As for the general case, since again $A$ and $B$ commute with exactly the same rank-one projections, we can jointly diagonalise them with respect to some orthonormal basis $\{e_j\}_{j=1}^n$, where $n = \dim H$:
	\begin{equation*}
		A = \left[\begin{matrix}
				\lambda_1 & 0 & \dots & 0 & 0 \\
				0 & \lambda_2 & \dots & 0 & 0 \\
				\vdots & & \ddots & & \vdots \\
				0 & 0 & \dots & \lambda_{n-1} & 0 \\
				0 & 0 & \dots & 0 & \lambda_n \\
			\end{matrix}\right]
		\quad \text{and} \quad
		B = \left[\begin{matrix}
				\mu_1 & 0 & \dots & 0 & 0 \\
				0 & \mu_2 & \dots & 0 & 0 \\
				\vdots & & \ddots & & \vdots \\
				0 & 0 & \dots & \mu_{n-1} & 0 \\
				0 & 0 & \dots & 0 & \mu_n \\
			\end{matrix}\right].
	\end{equation*}
	Of course, for any two distinct $i, j \in \{1,\dots, n\}$ we have the following equation for the strength functions:
	\begin{equation*}
		\Lambda(A, P) + \Lambda(A^\perp, P) = \Lambda(B, P) + \Lambda(B^\perp, P) \qquad (P \in P_1(\C\cdot e_i + \C\cdot e_j)),
	\end{equation*}
	which instantly implies 
	\begin{equation*}
		\left[\begin{matrix}
			\mu_i & 0 \\
			0 & \mu_j
		\end{matrix}\right]^\sim \cap \FiniteRank_1(\C\cdot e_i + \C\cdot e_j)
		=
		\left[\begin{matrix}
			\lambda_i & 0 \\
			0 & \lambda_j
		\end{matrix}\right]^\sim \cap \FiniteRank_1(\C\cdot e_i + \C\cdot e_j).
	\end{equation*}
	By the two-dimensional case this means that we have one of the following cases:
	\begin{itemize}
		\item $\lambda_i = \lambda_j$ and $\mu_i = \mu_j$,
		\item $\lambda_i \neq \lambda_j$ and either $\mu_i = \lambda_i$ and $\mu_j = \lambda_j$, or $\mu_i = 1-\lambda_i$ and $\mu_j = 1-\lambda_j$.
	\end{itemize}
	From here it is easy to conclude (iii).
\end{proof}

The commutant of an operator $T\in \Bdd(H)$ will be denoted by $T' := \{ S\in \Bdd(H) \colon ST=TS \}$, and more generally, if $\calM\subseteq \Bdd(H)$, then we set $\calM' := \cap\{ T' \colon T\in \calM\}$.
We shall use the notations $T'' := (T')'$ and $\calM'' := (\calM')'$ for the double commutants.

\begin{lemma}\label{lem:CoexSubset}
	For any $A,B \in \Effect(H)$ the following three assertions hold:
	\begin{itemize}
		\item[\textup{(a)}] If $A^\sim \subseteq B^\sim$, then $B\in A''$.
		\item[\textup{(b)}] If $\dim H \leq \aleph_0$ and $A^\sim \subseteq B^\sim$, then there exists a Borel function $f\colon [0,1] \to [0,1]$ such that $B = f(A)$.
		\item[\textup{(c)}] If $B$ is a convex combination of $A, A^\perp, 0$ and $I$, then $A^\sim \subseteq B^\sim$.
	\end{itemize}
\end{lemma}

\begin{proof}
	\emph{(a):} Assume that $C\in A'$. Our goal is to show $B\in C'$.
	We express $C$ in the following way:
	\begin{equation*}
		C = C_{\Re} + i C_{\Im}, \; C_{\Re} = \frac{C+C^*}{2}, \; C_{\Im} = \frac{C-C^*}{2i}
	\end{equation*}
	where $C_{\Re}$ and $C_{\Im}$ are self-adjoint (they are usually called the real and imaginary parts of $C$).
	Since $A$ is self-adjoint, $C^* \in A'$, hence $C_{\Re}, C_{\Im} \in A'$.
	Let $E_{\Re}$ and $E_{\Im}$ denote the projection-valued spectral measures of $C_{\Re}$ and $C_{\Im}$, respectively.
	By the spectral theorem (\cite[Theorem IX.2.2]{Conway}), Lemma \ref{lem:properties} and Corollary \ref{cor:ref}, we have $E_{\Re}(\Delta), E_{\Im}(\Delta) \in A^c \subseteq A^\sim \subseteq B^\sim$, therefore also $E_{\Re}(\Delta), E_{\Im}(\Delta) \in B'$ for all $\Delta \in \Borel_\R$, which gives $C\in B'$.

	\emph{(b):} This is an easy consequence of \cite[Proposition IX.8.1 and Lemma IX.8.7]{Conway}.

	\emph{(c):} If $A\sim C$, then also $A^\perp, 0$ and $I\sim C$. Hence by the convexity of $C^\sim$ we obtain $B\sim C$.
\end{proof}

Now, we are in the position to prove our first main result.

\begin{proof}[Proof of Theorem \ref{thm:CoexSet}]
	If $H$ is separable, then the equivalence (ii)$\iff$(iii) is straightforward by Lemma \ref{lem:WOTclosed}.
	For general $H$ the direction (i)$\Longrightarrow$(ii) is obvious, therefore we shall only prove (ii)$\Longrightarrow$(i), first in the separable, and then in the general case. 
	By Lemma \ref{lem:properties}, we may assume throughout the rest of the proof that $A$ and $B$ are non-scalar effects.
	We will denote the spectral subspace of a self-adjoint operator $T$ associated to a Borel set $\Delta\subseteq\R$ by $H_T(\Delta)$.
	
	\smallskip
	
	\emph{(ii)$\Longrightarrow$(i) in the separable case:}
	We split this part into two steps.
	
	\smallskip
	
	\emph{STEP 1:}
	Here, we establish two estimations, \eqref{eq:Estimation1} and \eqref{eq:Estimation2}, for the strength functions of $A$ and $B$ on certain subspaces of $H$.
	Let $\lambda_1, \lambda_2 \in \sigma(A), \lambda_1 \neq \lambda_2$ and $0 < \varepsilon < \tfrac{1}{2}|\lambda_1-\lambda_2|$.
	Then the spectral subspaces $H_1 = H_A\left( (\lambda_1-\varepsilon, \lambda_1+\varepsilon) \right)$ and $H_2 = H_A\left( (\lambda_2-\varepsilon, \lambda_2+\varepsilon) \right)$ are non-trivial and orthogonal.
	Set $H_3$ to be the orthogonal complement of $H_1\oplus H_2$, then the matrix of $A$ written in the orthogonal decomposition $H = H_1\oplus H_2 \oplus H_3$ is diagonal:
	\begin{equation*}
		A = \left[\begin{matrix}
			A_1 & 0 & 0 \\
			0 & A_2 & 0 \\
			0 & 0 & A_3
		\end{matrix}\right] \in \Bdd(H_1\oplus H_2 \oplus H_3).
	\end{equation*}
	Note that $H_3$ might be a trivial subspace.
	Since by Corollary \ref{cor:ref} $A$ and $B$ commute with exactly the same projections, the matrix of $B$ in $H = H_1\oplus H_2 \oplus H_3$ is also diagonal:
	\begin{equation*}
		B = \left[\begin{matrix}
			B_1 & 0 & 0 \\
			0 & B_2 & 0 \\
			0 & 0 & B_3
		\end{matrix}\right] \in \Bdd(H_1\oplus H_2 \oplus H_3).
	\end{equation*}
	At this point, let us emphasise that of course $H_j, A_j$ and $B_j$ ($j=1,2,3$) all depend on $\lambda_1, \lambda_2$ and $\varepsilon$, but in order to keep our notation as simple as possible, we will stick with these symbols.
	However, if at any point it becomes important to point out this dependence, we shall use for instance $B_j^{(\lambda_1, \lambda_2, \varepsilon)}$ instead of $B_j$.
	Similar conventions apply later on.
	
	Observe that by Corollary \ref{cor:kulon} we have
	\begin{equation*}
		\left[\begin{matrix}
					A_1 & 0 \\
					0 & A_2
		\end{matrix}\right]^\sim
		=
		\left[\begin{matrix}
					B_1 & 0 \\
					0 & B_2
		\end{matrix}\right]^\sim.
	\end{equation*}
	Now, we pick two arbitrary points $\mu_1\in \sigma(B_1)$ and $\mu_2\in \sigma(B_2)$.
	Then obviously, the following two subspaces are non-zero subspaces of $H_1$ and $H_2$, respectively:
	$$
		\widehat{H}_1 := (H_1)_{B_1}\big( (\mu_1-\varepsilon, \mu_1+\varepsilon) \big), \;\; 
		\widehat{H}_2 := (H_2)_{B_2}\big( (\mu_2-\varepsilon, \mu_2+\varepsilon) \big).
	$$
	Similarly as above, we have the following matrix forms where $\check{H}_j = H_j \ominus \widehat{H}_j$ $(j=1,2)$:
	$$
		B_1 = \left[\begin{matrix}
					\widehat{B}_1 & 0 \\
					0 & \check{B}_1
		\end{matrix}\right] \in \Bdd(\widehat{H}_1 \oplus \check{H}_1)
			\;\; \text{and} \;\;
		B_2 = \left[\begin{matrix}
					\widehat{B}_2 & 0 \\
					0 & \check{B}_2
		\end{matrix}\right] \in \Bdd(\widehat{H}_2 \oplus \check{H}_2)
	$$
	and
	$$
		A_1 = \left[\begin{matrix}
					\widehat{A}_1 & 0 \\
					0 & \check{A}_1
		\end{matrix}\right] \in \Bdd(\widehat{H}_1 \oplus \check{H}_1)
			\;\; \text{and} \;\;
		A_2 = \left[\begin{matrix}
					\widehat{A}_2 & 0 \\
					0 & \check{A}_2
		\end{matrix}\right] \in \Bdd(\widehat{H}_2 \oplus \check{H}_2).
	$$
	Note that $\check{H}_1$ or $\check{H}_2$ might be trivial subspaces.
	Again by Corollary \ref{cor:kulon}, we have 
	\begin{equation*}
		\left[\begin{matrix}
					\widehat{A}_1 & 0 \\
					0 & \widehat{A}_2
		\end{matrix}\right]^\sim
		=
		\left[\begin{matrix}
					\widehat{B}_1 & 0 \\
					0 & \widehat{B}_2
		\end{matrix}\right]^\sim.
	\end{equation*}
	Let us point out that by construction $\sigma(\widehat{A}_j) \subseteq [\lambda_j-\varepsilon,\lambda_j+\varepsilon]$ and $\sigma(\widehat{B}_j) \subseteq [\mu_j-\varepsilon,\mu_j+\varepsilon]$.
	Corollary \ref{cor:RankOneStrengthFunct} gives the following identity for the strength functions, where $\widehat{I}_j$ denotes the identity on $\widehat{H}_j$ $(j=1,2)$:
	\begin{align}\label{eq:WidehatStrength}
		\Lambda\left( \left[\begin{matrix}
					\widehat{A}_1 & 0 \\
					0 & \widehat{A}_2
		\end{matrix}\right], P \right) &+ 
		\Lambda\left( \left[\begin{matrix}
					\widehat{I}_1 - \widehat{A}_1 & 0 \\
					0 & \widehat{I}_2 - \widehat{A}_2
		\end{matrix}\right], P \right) \nonumber \\
		& =
		\Lambda\left( \left[\begin{matrix}
					\widehat{B}_1 & 0 \\
					0 & \widehat{B}_2
		\end{matrix}\right], P \right) + 
		\Lambda\left( \left[\begin{matrix}
					\widehat{I}_1 - \widehat{B}_1 & 0 \\
					0 & \widehat{I}_2 - \widehat{B}_2
		\end{matrix}\right], P \right) \quad \left(\forall\; P\in\Proj_1\left(\widehat{H}_1\oplus\widehat{H}_2\right)\right).
	\end{align}
	Define
	\begin{equation*}
		\Theta\colon \R \to [0,1], \quad \Theta(t) = 
		\left\{ \begin{matrix}
			0 & \text{if} \; t < 0 \\
			t & \text{if} \; 0 \leq t \leq 1 \\
			1& \text{if} \; 1 < t
		\end{matrix} \right.,
	\end{equation*}
	and notice that we have the following two estimations for all rank-one projections $P$:
	\begin{align}\label{eq:Estimation1}
		\Lambda&\left( \left[\begin{matrix}
					\Theta(\lambda_1-\varepsilon)\widehat{I}_1 & 0 \\
					0 & \Theta(\lambda_2-\varepsilon)\widehat{I}_2
		\end{matrix}\right], P \right) + 
		\Lambda\left( \left[\begin{matrix}
					\Theta(1-\lambda_1-\varepsilon)\widehat{I}_1 & 0 \\
					0 & \Theta(1-\lambda_2-\varepsilon)\widehat{I}_2
		\end{matrix}\right], P \right)  \nonumber \\
		& \leq \text{the expression in \eqref{eq:WidehatStrength}} \nonumber \\
		& \leq \Lambda\left( \left[\begin{matrix}
					\Theta(\lambda_1+\varepsilon)\widehat{I}_1 & 0 \\
					0 & \Theta(\lambda_2+\varepsilon)\widehat{I}_2
		\end{matrix}\right], P \right) + 
		\Lambda\left( \left[\begin{matrix}
					\Theta(1-\lambda_1+\varepsilon)\widehat{I}_1 & 0 \\
					0 & \Theta(1-\lambda_2+\varepsilon)\widehat{I}_2
		\end{matrix}\right], P \right)
	\end{align}
	and 
	\begin{align}\label{eq:Estimation2}
		\Lambda&\left( \left[\begin{matrix}
					\Theta(\mu_1-\varepsilon)\widehat{I}_1 & 0 \\
					0 & \Theta(\mu_2-\varepsilon)\widehat{I}_2
		\end{matrix}\right], P \right) + 
		\Lambda\left( \left[\begin{matrix}
					\Theta(1-\mu_1-\varepsilon)\widehat{I}_1 & 0 \\
					0 & \Theta(1-\mu_2-\varepsilon)\widehat{I}_2
		\end{matrix}\right], P \right) \nonumber \\
		& \leq \text{the expression in \eqref{eq:WidehatStrength}} \nonumber \\
		& \leq \Lambda\left( \left[\begin{matrix}
					\Theta(\mu_1+\varepsilon)\widehat{I}_1 & 0 \\
					0 & \Theta(\mu_2+\varepsilon)\widehat{I}_2
		\end{matrix}\right], P \right) + 
		\Lambda\left( \left[\begin{matrix}
					\Theta(1-\mu_1+\varepsilon)\widehat{I}_1 & 0 \\
					0 & \Theta(1-\mu_2+\varepsilon)\widehat{I}_2
		\end{matrix}\right], P \right).
	\end{align}
	Note that the above estimations hold for any arbitrarily small $\varepsilon$ and for all suitable choices of $\mu_1$ and $\mu_2$ (which of course depend on $\varepsilon$).
	
	\smallskip
	
	\emph{STEP 2:}
	Here we show that $B\in\{A,A^\perp\}$.
	Let us define the following set that depends only on $\lambda_j$:
	\begin{equation*}
		\calC_j = \calC_j^{(\lambda_j)} := \bigcap \left\{ \sigma\left( B_j^{(\lambda_1, \lambda_2, \varepsilon)} \right) \colon 0 < \varepsilon < \tfrac{1}{2}|\lambda_1 - \lambda_2| \right\}
		= \bigcap \left\{ \sigma\left( B|_{H_A((\lambda_j-\varepsilon,\lambda_j+\varepsilon))} \right) \colon 0 < \varepsilon \right\}
		 \quad (j=1,2).
	\end{equation*}	
	Notice that as this set is an intersection of monotonically decreasing (as $\varepsilon \searrow 0$), compact, non-empty sets, it must contain at least one element.
	Also, observe that if $\mu_1\in\calC_1$ and $\mu_2\in\calC_2$, then \eqref{eq:Estimation1} and \eqref{eq:Estimation2} hold for all $\varepsilon>0$.
	
	We proceed with proving that either $\calC_1 = \{\lambda_1\}$ and $\calC_2 = \{\lambda_2\}$, or $\calC_1 = \{1-\lambda_1\}$ and $\calC_2 = \{1-\lambda_2\}$ hold.
	Fix two arbitrary elements $\mu_1 \in \calC_1$ and $\mu_2 \in \calC_2$, and assume that neither $\lambda_1 = \mu_1$ and $\lambda_2 = \mu_2$, nor $\lambda_1 + \mu_1 = \lambda_2 + \mu_2 = 1$ hold. 
	From here our aim is to get a contradiction.
	As we showed in the proof of Lemma \ref{lem:BlockDiagSca}, there exists an $\alpha_0 \in \left(0,\tfrac{\pi}{2}\right)$ such that we have
	\begin{align*}
		\frac{1}{ \left(\tfrac{1}{\lambda_1}\right)\cdot \cos^2\alpha_0 + \left(\tfrac{1}{\lambda_2}\right)\cdot \sin^2\alpha_0 } &+ \frac{1}{ \left(\tfrac{1}{1-\lambda_1}\right)\cdot \cos^2\alpha_0 + \left(\tfrac{1}{1-\lambda_2}\right)\cdot \sin^2\alpha_0 } \\
		&\neq
		\frac{1}{ \left(\tfrac{1}{\mu_1}\right)\cdot \cos^2\alpha_0 + \left(\tfrac{1}{\mu_2}\right)\cdot \sin^2\alpha_0 } + \frac{1}{  \left(\tfrac{1}{1-\mu_1}\right)\cdot \cos^2\alpha_0 + \left(\tfrac{1}{1-\mu_2}\right)\cdot \sin^2\alpha_0 }\nonumber
	\end{align*}
	where we interpret both sides as in \eqref{eq:DiagStrength}.
	Notice that both summands on both sides depend continuously on $\lambda_1,\lambda_2,\mu_1$ and $\mu_2$.
	Therefore there exists an $\varepsilon>0$ small enough and a rank-one projection $P = P_{\cos\alpha_0 \widehat{x}_1+\sin\alpha_0 \widehat{x}_2}$, with $\widehat{x}_1\in\widehat{H}_1, \widehat{x}_2\in\widehat{H}_2, \|\widehat{x}_1\|=\|\widehat{x}_2\|=1$, such that the closed intervals bounded by the right- and left-hand sides of \eqref{eq:Estimation1}, and those of \eqref{eq:Estimation2} are disjoint -- which is a contradiction.
	
	Observe that as we can do the above for any two disjoint elements of the spectrum $\sigma(A)$, we can conclude that one of the following possibilities occur:
	\begin{equation}\label{eq:FirstCase}
		\{\lambda\} = \bigcap \left\{ \sigma\left( B|_{H_A((\lambda-\varepsilon, \lambda+\varepsilon))} \right) \colon \varepsilon > 0 \right\} \qquad (\lambda\in\sigma(A))
	\end{equation}
	or
	\begin{equation}\label{eq:SecondCase}
		\{1-\lambda\} = \bigcap \left\{ \sigma\left( B|_{H_A((\lambda-\varepsilon, \lambda+\varepsilon))} \right) \colon \varepsilon > 0 \right\} \qquad (\lambda\in\sigma(A)).
	\end{equation}

	From here, we show that \eqref{eq:FirstCase} implies $A=B$, and \eqref{eq:SecondCase} implies $B=A^\perp$.
	As the latter can be reduced to the case \eqref{eq:FirstCase}, by considering $B^\perp$ instead of $B$, we may assume without loss of generality that \eqref{eq:FirstCase} holds.
	By Lemma \ref{lem:CoexSubset} and \cite[Theorem IX.8.10]{Conway}, there exists a function $f\in L^\infty(\mu)$, where $\mu$ is a scalar-valued spectral measure of $A$, such that $B = f(A)$.
	Moreover, we have $B = A$ if and only if $f(\lambda) = \lambda$ $\mu$-a.e, so we only have to prove the latter equation.
	Let us fix an arbitrarily small number $\delta > 0$.
	By the spectral mapping theorem (\cite[Theorem IX.8.11]{Conway}) and \eqref{eq:FirstCase} we notice that for every $\lambda \in \sigma(A)$ there exists an $0 < \varepsilon_{\lambda} < \delta$ such that
	\begin{equation}\label{eq:SpectrumCloseToLambda}
		\mu-\mathrm{essran} \left( f|_{(\lambda-\varepsilon_{\lambda},\lambda+\varepsilon_{\lambda})} \right) = \sigma\left(B|_{H_A((\lambda-\varepsilon_{\lambda},\lambda+\varepsilon_{\lambda}))} \right) \subseteq (\lambda-\delta,\lambda+\delta),
	\end{equation}
	where $\mu-\mathrm{essran}$ denotes the essential range of a function with respect to $\mu$ (see \cite[Example IX.2.6]{Conway}).
	Now, for every $\lambda\in\sigma(A)$ we fix such an $\varepsilon_{\lambda}$.
	Clearly, the intervals $\{(\lambda-\varepsilon_{\lambda},\lambda+\varepsilon_{\lambda}) \colon \lambda\in\sigma(A)\}$ cover the whole spectrum $\sigma(A)$, which is a compact set. 
	Therefore we can find finitely many of them, let's say $\lambda_1, \dots, \lambda_n$ so that 
	\begin{equation*}
		\sigma(A) 
		\subseteq \bigcup_{j=1}^{n} \left(\lambda_j-\varepsilon_{\lambda_j},\lambda_j+\varepsilon_{\lambda_j} \right).
	\end{equation*}	
	Finally, we define the function
	\begin{equation*}
		h(\lambda) = \lambda_j, \text{ where } |\lambda-\lambda_{i}| \geq \varepsilon_{\lambda_i} \text{ for all } 1\leq i < j \text{ and } |\lambda-\lambda_{j}| < \varepsilon_{\lambda_j}.
	\end{equation*}
	By definition we have $\| h - \mathrm{id}_{\sigma(A)} \|_{\infty} \leq \delta$ where the $\infty$-norm is taken with respect to $\mu$ and $\mathrm{id}_{\sigma(A)}(\lambda) = \lambda$ $(\lambda\in\sigma(A))$.
	But notice that by \eqref{eq:SpectrumCloseToLambda} we also have $\| h - f \|_{\infty} \leq \delta$, and hence $\| f - \mathrm{id}_{\sigma(A)} \|_\infty \leq 2\delta$.
	As this inequality holds for all positive $\delta$, we actually get that $f(\lambda) = \lambda$ for $\mu$-a.e.~$\lambda$.
	
	\smallskip
	
	\emph{(ii)$\Longrightarrow$(i) in the non-separable case:}
	It is well-known that there exists an orthogonal decomposition $H = \oplus_{i\in\mathcal{I}} H_i$ such that each $H_i$ is a non-trivial, separable, invariant subspace of $A$, see for instance \cite[Proposition IX.4.4]{Conway}.
	Since $A$ and $B$ commute with exactly the same projections, both are diagonal with respect to the decomposition $H = \oplus_{i\in\mathcal{I}} H_i$:
	\begin{equation*}
		A = \oplus_{i\in \mathcal{I}} A_i \;\; \text{and} \;\; B = \oplus_{i\in \mathcal{I}} B_i \in \Effect(\oplus_{i\in \mathcal{I}} H_i).
	\end{equation*}
	By Corollary \ref{cor:kulon} we have $A_i^\sim = B_i^\sim$ for all $i\in\mathcal{I}$, therefore the separable case implies
	\begin{equation*}
		\text{either} \; A_i = B_i, \;\; \text{or} \; B_i = A_i^\perp, \;\; \text{or} \; A_i, B_i \in \Sca(H_i) \qquad (i\in\mathcal{I}).
	\end{equation*}
	Without loss of generality we may assume from now on that there exists an $i_0\in\mathcal{I}$ so that $A_{i_0}$ is not a scalar effect.
	(In case all of them are scalar, we simply combine two subspaces $H_{i_1}$ and $H_{i_2}$ so that $\sigma(A_{i_1})\neq\sigma(A_{i_2})$).
	This implies either $A_{i_0} = B_{i_0}$, or $B_{i_0} = A_{i_0}^\perp$.
	By considering $B^\perp$ instead of $B$ if necessary, we may assume from now on that $A_{i_0} = B_{i_0}$ holds.	
	
	Finally, let $i_1 \in \mathcal{I} \setminus \{i_0\}$ be arbitrary, and let us consider the orthogonal decomposition $H = \oplus_{i\in \mathcal{I}\setminus\{i_0,i_1\}} H_i \oplus K$ where $K = H_{i_0}\oplus H_{i_1}$.
	Similarly as above, we obtain either $A_{i_0}\oplus A_{i_1} = B_{i_0} \oplus B_{i_1}$, or $B_{i_0}\oplus B_{i_1} = A_{i_0}^\perp \oplus A_{i_1}^\perp$, but since $A_{i_0} = B_{i_0}$, we must have $A_{i_1} = B_{i_1}$. 
	As this holds for arbitrary $i_1$, the proof is complete.
\end{proof}

Now, we are in the position to give an alternative proof of Moln\'ar's theorem which also extends to the two-dimensional case.

\begin{proof}[Proof of Theorem \ref{thm:2dM} and Moln\'ar's theorem]
	By (a) of Lemma \ref{lem:properties} and \eqref{eq:coex} we obtain $\phi(\Sca(H)) = \Sca(H)$, moreover, the property \eqref{eq:ord} implies the existence of a strictly increasing bijection $g\colon [0,1] \to [0,1]$ such that $\phi(\lambda I) = g(\lambda) I$ for every $\lambda\in [0,1]$.
	By Theorem \ref{thm:CoexSet} we conclude
	\begin{equation*}
		\phi(A^\perp) = \phi(A)^\perp \qquad (A\in \Effect(H)\setminus \Sca(H)).
	\end{equation*}
	We only have to show that the same holds for scalar operators, because then the theorem is reduced to Ludwig's theorem.
	For any effect $A$ and any set of effects ${\cal S}$ let us define the following sets $A^\leq := \{B\in \Effect(H)\colon A\leq B\}$, $A^\geq := \{B\in \Effect(H)\colon A\geq B\}$ and ${\cal S}^\perp := \{B^\perp \colon B\in {\cal S}\}$.
	Observe that for any $s,t\in [0,1]$ we have
	\begin{equation}\label{eq:sItI}
		\left((sI)^\leq \cap (tI)^\geq \setminus \Sca(H)\right)^\perp = (sI)^\leq \cap (tI)^\geq \setminus \Sca(H) \neq \emptyset
	\end{equation}
	if and only if $t=1-s$ and $s< \tfrac{1}{2}$.
	Thus for all $s < \tfrac{1}{2}$  we obtain 
	\begin{align*}
		\emptyset \neq \left((g(s)I)^\leq \cap (g(1-s)I)^\geq \setminus \Sca(H)\right)^\perp = \phi\left(\left((sI)^\leq \cap ((1-s)I)^\geq \setminus \Sca(H)\right)^\perp\right) \\
		= \phi\left((sI)^\leq \cap ((1-s)I)^\geq \setminus \Sca(H)\right) = (g(s)I)^\leq \cap (g(1-s)I)^\geq \setminus \Sca(H),
	\end{align*}
	which by \eqref{eq:sItI} implies $g(1-s) = 1-g(s)$ and $g(s)< \tfrac{1}{2}$, therefore we indeed have \eqref{eq:perp} for every effect. 
\end{proof}

%-------------------------------------------------------------------------------------------------------

\section{Proof of Theorem \ref{thm:main} in two dimensions}

In this section we prove our other main theorem for qubit effects.
In order to do that we need to prove a few preparatory lemmas. 
We start with a characterisation of rank-one projections in terms of coexistence.

\begin{lemma}\label{lem:2dprojchar}
	For any $A\in \Effect(\C^2)$ the following are equivalent:
	\begin{itemize}
		\item[\textup{(i)}] there are no effects $B\in \Effect(\C^2)$ such that $B^\sim\subsetneq A^\sim$,
		\item[\textup{(ii)}] $A \in \Proj_1(\C^2)$.
	\end{itemize}
\end{lemma}

\begin{proof}
	The case when $A\in \Sca(\C^2)$ is trivial, therefore we may assume otherwise throughout the proof.
	
	\emph{(i)$\Longrightarrow$(ii):}
	Suppose that $A \notin \Proj_1(\C^2)$, then by Corollary \ref{cor:0-neighbourhood} there exists an $\varepsilon>0$ such that $\{C\in \Effect(\C^2) \colon C \leq \varepsilon I\} \subseteq A^\sim$.
	Let $B\in\Proj_1(\C^2)\cap A^c$, then we have $B^\sim = B^c = A^c\subseteq A^\sim$.
	But it is very easy to find a $C\in \Effect(\C^2)$ such that $C \leq \varepsilon I$ and $C\notin B^c$, therefore we conclude $B^\sim\subsetneq A^\sim$.
	
	\emph{(ii)$\Longrightarrow$(i):}
	If $A \in \Proj_1(\C^2)$, $B\in \Effect(\C^2)$ and $B^\sim\subsetneq A^\sim$, then also $B^c\subsetneq A^c$, which is impossible.
\end{proof}

Note that the above statement does not hold in higher dimensions, see the final section of this paper for more details.
We continue with a characterisation of rank-one and ortho-rank-one qubit effects in terms of coexistence.

\begin{lemma}\label{lem:F1order}
	Let $A\in \Effect(\C^2)\setminus \Sca(\C^2)$. Then the following are equivalent:
	\begin{itemize}
		\item[\textup{(i)}] $A$ or $A^\perp \in\FiniteRank_1(\C^2)\setminus\Proj_1(\C^2)$,
		\item[\textup{(ii)}] There exists at least one $B\in \Effect(\C^2)$ such that $B^\sim \subsetneq A^\sim$, and for every such pair of effects $B_1, B_2$ we have either $B_1^\sim \subseteq B_2^\sim$, or $B_2^\sim \subseteq B_1^\sim$.
	\end{itemize}
Moreover, if (i) holds, i.e.~$A$ or $A^\perp = tP$ with $P\in\Proj_1(\C^2)$ and $0<t<1$, then we have $B^\sim \subseteq A^\sim$ if and only if $B$ or $B^\perp = sP$ with some $t\leq s\leq 1$.
\end{lemma}

\begin{proof}
	First, notice that by Theorem \ref{thm:CoexSet} and Lemma \ref{lem:CoexSubset} (c) we have
	$$
	(sP)^\sim \subseteq (tP)^\sim \quad\iff\quad t\leq s \qquad (P\in\Proj_1(\C^2),\; t, s \in (0,1]).
	$$ 

	\emph{(i)$\Longrightarrow$(ii):}
	If we have $B^\sim \subseteq (tP)^\sim$ with some rank-one projection $P$, $t \in (0,1]$ and qubit effect $B$, then by Lemma \ref{lem:CoexSubset} (b) we obtain $P\in B^c$ and $B\notin \Sca(\C^2)$.
	Furthermore, since $B^\sim\cap\FiniteRank_1(\C^2) \subseteq (tP)^\sim\cap\FiniteRank_1(\C^2)$, by Corollary \ref{cor:RankOneStrengthFunct} we obtain
	$$
	T_B(\alpha) \leq T_{tP}(\alpha) \qquad (0\leq \alpha \leq \tfrac{\pi}{2}),
	$$ 
	where we use the notation from the proof of Lemma \ref{lem:BlockDiagSca}.
	Thus, the discontinuity of $T_{tP}(\alpha)$ at either $\alpha = 0$, or $\alpha = \tfrac{\pi}{2}$, implies the discontinuity of $T_B(\alpha)$ at the same $\alpha$. 
	Whence we conclude either $B=sP$, or $B=I-sP$ with some $t\leq s\leq 1$.

	\emph{(ii)$\Longrightarrow$(i):}
	By Lemma \ref{lem:2dprojchar}, (ii) cannot hold for elements of $\Proj_1(\C^2)$, so we only have to check that if $A, A^\perp \notin \FiniteRank_1(\C^2) \cup \Sca(\C^2)$, then (ii) fails.
	Suppose that the spectral decomposition of $A$ is $\lambda_1 P + \lambda_2 P^\perp$ where $1 > \lambda_1>\lambda_2 > 0$.
	Then by Lemma \ref{lem:CoexSubset} (c) we find that $\left(\lambda_1 P\right)^\sim \subseteq A^\sim$ and $\left((1-\lambda_2) P^\perp\right)^\sim \subseteq A^\sim$ (see Figure \ref{fig:parallelogram}), but by the previous part neither $(\lambda_1 P)^\sim \subseteq \left((1-\lambda_2) P^\perp\right)^\sim$, nor $\left((1-\lambda_2) P^\perp\right)^\sim \subseteq (\lambda_1 P)^\sim$ holds.
\end{proof}

	\begin{figure}[h!]
	\begin{center}
		\begin{tikzpicture}[scale=2]
		%\draw[fill,lightgray] (0,0)--(-0.4,0.8)--(0,2)--(0.4,1.2)--(0,0);
		\draw[fill] (0,0) circle (0.02);
		\node[below] at (0,0) {$0$};
		\draw[fill] (-1,1) circle (0.02);
		\node[left] at (-1,1) {$P$};
		\draw[fill] (1,1) circle (0.02);
		\node[right] at (1,1) {$P^\perp$};
		\draw[fill] (0,2) circle (0.02);
		\node[above] at (0,2) {$I$};
		\draw (0,0)--(1,1)--(0,2)--(-1,1)--(0,0);
		\draw[fill] (-0.4,0.8) circle (0.02);
		\node[left] at (-0.4,0.85) {$A$};
		\draw[fill] (0.4,1.2) circle (0.02);
		\node[right] at (0.4,1.2) {$A^\perp$};
		\draw[dashed] (-0.4,0.8)--(-0.6,0.6);
		\draw[fill] (-0.6,0.6) circle (0.02);
		\node[left] at (-0.6,0.55) {$\lambda_1 P$};
		\draw[dashed] (0.4,1.2)--(0.8,0.8);
		\draw[fill] (0.8,0.8) circle (0.02);
		\node[below] at (1.2,0.8) {$(1-\lambda_2) P^\perp$};
		\end{tikzpicture}
		\caption{The figure shows all effects commuting with $A \in \Effect(\C^2)\setminus \Sca(\C^2)$, whose spectral decomposition is $A = \lambda_1 P + \lambda_2 P^\perp$ with $1>\lambda_1>\lambda_2>0$.}
		\label{fig:parallelogram}
	\end{center}
	\end{figure}
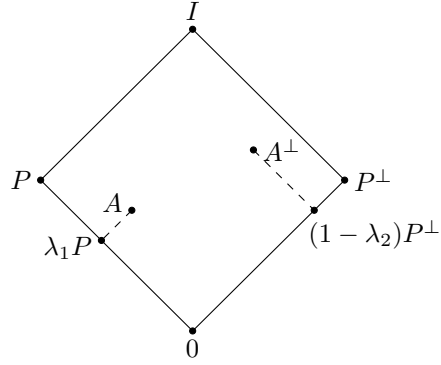

For a visualisation of $(tP)^\sim\cap\FiniteRank_1(\C^2)$ see Section \ref{sec:Visual}.
Before we proceed with the proof of Theorem \ref{thm:main} for qubit effects, we need a few more lemmas about rank-one projections acting on $\C^2$.

\begin{lemma}\label{lem:P0}
	For all $P, Q\in \Proj_1(\C^2)$ we have
	\begin{equation}\label{eq:P0}
		\|P-Q\|^2 = -\det(P-Q) = 1-\tr PQ = 1 - \|P^\perp-Q\|^2.
	\end{equation}
\end{lemma}

\begin{proof}
	Since $\tr (P-Q) = 0$, the eigenvalues of the self-adjoint operator $P-Q$ are $\lambda$ and $-\lambda$ with some $\lambda \geq 0$.
	Hence we have $\|P-Q\|^2 = -\det(P-Q)$.
	Applying a unitary similarity if necessary, we may assume without loss of generality that $(1,0) \in \Image P$.
	Obviously, there exist $0\leq \vartheta\leq \tfrac{\pi}{2}$ and $0\leq \mu \leq 2\pi$ such that $(\cos\vartheta, e^{i\mu}\sin\vartheta) \in \Image Q$.
	Thus the matrix forms of $P$ and $Q$ in the standard basis are
	\begin{equation}\label{eq:mxformPQ1}
		P = P_{(1,0)}
		= \left[
		\begin{matrix}
			1 \\ 
			0
		\end{matrix}
		\right] \cdot
		\left[
		\begin{matrix}
			1 \\
			0
		\end{matrix}
		\right]^*
		= \left[
		\begin{matrix}
			1 & 0 \\
			0 & 0
		\end{matrix}
		\right]
	\end{equation}
	and
	\begin{equation}\label{eq:mxformPQ2}
		Q = P_{(\cos\vartheta, e^{i\mu}\sin\vartheta)}
		= \left[
		\begin{matrix}
			\cos\vartheta \\
			e^{i\mu}\sin\vartheta
		\end{matrix}
		\right] \cdot
		\left[
		\begin{matrix}
			\cos\vartheta \\
			e^{i\mu}\sin\vartheta
		\end{matrix}
		\right]^*
		=
		\left[
		\begin{matrix}
			\cos^2\vartheta & e^{-i\mu}\cos\vartheta\sin\vartheta \\
			e^{i\mu}\cos\vartheta\sin\vartheta & \sin^2\vartheta
		\end{matrix}
		\right],
	\end{equation}
	where we used the notation of the Busch--Gudder theorem.
	Now, an easy calculation gives us $\det(P-Q) = -\sin^2\vartheta$ and $\tr PQ = \cos^2\vartheta$. 
	Hence the second equation in \eqref{eq:P0} is proved, and the third one follows from $\tr P^\perp Q = 1 - \tr PQ$.
\end{proof}

For $P\in\Proj_1(\C^2)$ and $s\in [0,1]$, let us use the following notation:
$$
	\calM_{P,s} := \left\{ Q\in\Proj_1(\C^2) \colon \|P-Q\| = s \right\}.
$$
Next, we examine this set.

\begin{lemma}\label{lem:P1}
	For all $P\in \Proj_1(\C^2)$ the following statements are equivalent:
	\begin{itemize}
		\item[\textup{(i)}] $s = \sin\tfrac{\pi}{4}$,
		\item[\textup{(ii)}] there exists an $R\in \calM_{P,s}$ such that $R^\perp \in \calM_{P,s}$,
		\item[\textup{(iii)}] for all $R\in \calM_{P,s}$ we have also $R^\perp \in \calM_{P,s}$.
	\end{itemize}
\end{lemma}

\begin{proof}
	One could use the Bloch representation (see Section \ref{sec:Visual}), however, let us give here a purely linear algebraic proof.
	Note that for any $R_1, R_2 \in \Proj_1(\C^2)$ we have $\|R_1-R_2\| = 1$ if and only if $R_2 = R_1^\perp$.
	Without loss of generality we may assume that $P$ has the matrix form of \eqref{eq:mxformPQ1}.
	Then for any $0\leq \vartheta \leq \tfrac{\pi}{2}$ and $R_1, R_2 \in \calM_{P,\sin\vartheta}$ we have
	$$
		R_1
		= \left[
		\begin{matrix}
			\cos^2\vartheta & e^{-i\mu_1}\cos\vartheta\sin\vartheta \\
			e^{i\mu_1}\cos\vartheta\sin\vartheta & \sin^2\vartheta
		\end{matrix}
		\right]
		\quad\text{and}\quad
		R_2
		= \left[
		\begin{matrix}
			\cos^2\vartheta & e^{-i\mu_2}\cos\vartheta\sin\vartheta \\
			e^{i\mu_2}\cos\vartheta\sin\vartheta & \sin^2\vartheta
		\end{matrix}
		\right]
	$$
	with some $\mu_1,\mu_2\in \R$.
	Hence, we get
	\begin{align*}
		\|R_1-R_2\| & = \sqrt{1-\tr R_1R_2} = \sqrt{\sin^2\vartheta \cos^2\vartheta (2-e^{i(\mu_1-\mu_2)}-e^{i(\mu_2-\mu_1)})} \\
		& = |e^{i\mu_1}-e^{i\mu_2}| \cos\vartheta \sin\vartheta = \tfrac{1}{2} |e^{i\mu_1}-e^{i\mu_2}| \sin(2 \vartheta).
	\end{align*}
	Notice that the right-hand side is always less than or equal to 1.
	Moreover, for any $\mu_1\in\R$ there exist a $\mu_2\in\R$ such that $\|R_1-R_2\|=1$ if and only if $\vartheta = \tfrac{\pi}{4}$.
	This completes the proof.
\end{proof}

\begin{lemma}\label{lem:P2}
	Let $P,Q\in \Proj_1(\C^2)$ and $s,t\in(0,1)$. 
	Then the following are equivalent:
	\begin{itemize}
		\item[\textup{(i)}] $tP\sim sQ$
		\item[\textup{(ii)}] either $Q=P$, or $Q=P^\perp$, or 
		$$
			s \leq \frac{1}{\tfrac{1}{1-t}\|P^\perp-Q\|^2+\|P-Q\|^2}.
		$$
	\end{itemize}
\end{lemma}

\begin{proof}
	The case when $Q\in\{P,P^\perp\}$ is trivial, so from now on we assume otherwise.
	Recall that two rank-one effects with different images are coexistent if and only if their sum is an effect, see \cite[Lemma 2]{Mol0}.
	Therefore, (i) is equivalent to $I-tP-sQ \geq 0$.
	Since $\tr(I-tP-sQ) = 2-t-s > 0$, the latter is further equivalent to $\det(I-tP-sQ) \geq 0$.
	Without loss of generality we may assume that $P$ and $Q$ have the matrix forms written in \eqref{eq:mxformPQ1} and \eqref{eq:mxformPQ2} with $0<\vartheta<\tfrac{\pi}{2}$. 
	Then a calculation gives
	$$
		\det(I-tP-sQ) = s(t-1)\sin^2\vartheta - s\cos^2\vartheta +1-t = 1-t-s+ts\|P-Q\|^2.
	$$
	From the latter we get that $\det(I-tP-sQ) \geq 0$ holds if and only if
	$$
		s\leq \frac{1-t}{1-t\|P-Q\|^2},
	$$
	which, by \eqref{eq:P0} is equivalent to (ii).
\end{proof}

Note that we have
$$
	0 < \frac{1}{\tfrac{1}{1-t}\|P^\perp-Q\|^2+\|P-Q\|^2} < 1 \qquad (t\in (0,1), P,Q\in \Proj_1(\C^2), Q\notin \{P,P^\perp\}).
$$
We need one more lemma.

\begin{lemma}\label{lem:P3}
	Let $P,Q\in \Proj_1(\C^2)$. 
	Then there exists a projection $R\in \Proj_1(\C^2)$ such that
	$$
		\|P-R\| = \|Q-R\| = \sin\tfrac{\pi}{4}.
	$$
\end{lemma}

\begin{proof}
	Again, one could use the Bloch representation, however, let us give here a purely linear algebraic proof.
	We may assume without loss of generality that $P$ and $Q$ are of the form \eqref{eq:mxformPQ1} and \eqref{eq:mxformPQ2}.
	Then for any $z\in\C$, $|z|=1$ the rank-one projection
	$$
		R = 
		\frac{1}{\sqrt{2}}\left[
		\begin{matrix}
			1 \\
			z
		\end{matrix}
		\right] \cdot
		\left( \frac{1}{\sqrt{2}}\left[
		\begin{matrix}
			1\\
			z
		\end{matrix}
		\right]\right)^*
		= \frac{1}{2}\left[
		\begin{matrix}
			1 & \overline{z} \\
			z & 1
		\end{matrix}
		\right]
	$$
	satisfies $\|P-R\| = \sin\tfrac{\pi}{4}$.
	In order to complete the proof we only have to find a $z$ with $|z|=1$ such that $\tr RQ = \tfrac{1}{2}$, which is an easy calculation.
	Namely, we find that $z=ie^{i\mu}$ is a suitable choice.
\end{proof}

Now, we are in the position to prove our second main result in the low-dimensional case.

\begin{proof}[Proof of Theorem \ref{thm:main} in two dimensions]
	The proof is divided into the following three steps: 
	\begin{itemize}
		\item[1] we show some basic properties of $\phi$, in particular, that it preserves commutativity in both directions,
		\item[2] we show that $\phi$ maps pairs of rank-one projections with distance $\sin\tfrac{\pi}{4}$ into pairs of rank-one projections with the same distance,
		\item[3] we finish the proof by examining how $\phi$ acts on rank-one projections and rank-one effects.
	\end{itemize} 
	
	\smallskip

	\emph{STEP 1:} First of all, the properties of $\phi$ imply 
	\begin{equation*}
		\phi(A)^\sim = \phi(A^\sim) \quad (A\in \Effect(\C^2)),
	\end{equation*}
	and
	\begin{equation*}
		B^\sim \subseteq A^\sim \;\;\iff\;\; \phi(B)^\sim \subseteq \phi(A)^\sim \quad (A,B\in \Effect(\C^2)).
	\end{equation*}
	Hence, it is straightforward from Lemma \ref{lem:properties} that there exists a bijection $g\colon [0,1] \to [0,1]$ such that
	\begin{equation}\label{eq:Scafixed}
		\phi(t I) = g(t)I \quad (t\in[0,1]).
	\end{equation} 
	Also, by Lemma \ref{lem:2dprojchar} we easily infer 
	\begin{equation*}
		\phi(\Proj_1(\C^2)) = \Proj_1(\C^2),
	\end{equation*}
	thus, in particular, we get
	\begin{equation*}
		\phi(P^c) = \phi(P^\sim) = \phi(P)^\sim = \phi(P)^c \quad (P\in \Proj_1(\C^2)).
	\end{equation*}
	By Theorem \ref{thm:CoexSet} we also obtain
	\begin{equation*}
		\phi(A^\perp) = \phi(A)^\perp \quad (A\in \Effect(\C^2)\setminus \Sca(\C^2)).
	\end{equation*}
	Now, we observe that $\phi$ preserves commutativity in both directions. Indeed we have the following for every $A,B \in \Effect(\C^2)\setminus \Sca(\C^2)$:
	\begin{align*}
		AB = BA &\,\iff\, A^\sim\cap\Proj_1(\C^2) = B^\sim\cap\Proj_1(\C^2) = \{P,P^\perp\} \text{ for some } P\in\Proj_1(\C^2)  \\
		&\,\iff\, \phi(A)^\sim\cap\Proj_1(\C^2) = \phi(B)^\sim\cap\Proj_1(\C^2) = \{Q,Q^\perp\} \text{ for some } Q\in\Proj_1(\C^2) \\
		&\,\iff\, \phi(A)\phi(B) = \phi(B)\phi(A).
	\end{align*}
	Note that we easily get the same conclusion using \eqref{eq:Scafixed} if any of the two effects is a scalar effect.

	Next, notice that Lemma \ref{lem:F1order} implies 
	\begin{equation*}
		A \,\text{or}\, A^\perp \in \FiniteRank_1(\C^2)\setminus\Proj_1(\C^2)
		\;\;\iff\;\; \phi(A) \,\text{or}\, \phi(A)^\perp \in \FiniteRank_1(\C^2)\setminus\Proj_1(\C^2).
	\end{equation*}
	Therefore, by interchanging the $\phi$-images of $tP$ and $I-tP$ for some $0<t<1$ and $P\in\Proj_1(\C^2)$, we may assume without loss of generality that
	\begin{equation*}
		\phi\left(\FiniteRank_1(\C^2)\setminus\Proj_1(\C^2)\right) = \FiniteRank_1(\C^2)\setminus\Proj_1(\C^2).
	\end{equation*}
	Hence we obtain the following for all rank-one projections $P$:
	\begin{equation*}
		\phi\left(\{tP, tP^\perp\colon 0<t\leq 1\}\right) = \phi\left(P^c\cap\FiniteRank_1(\C^2)\right) 
		= \phi(P)^c\cap \FiniteRank_1(\C^2) = \{t\phi(P), t\phi(P)^\perp\colon 0<t\leq 1\}.
	\end{equation*}
	Thus, again by interchanging the $\phi$-images of $P$ and $P^\perp$ for some $P\in\Proj_1(\C^2)$, and using Lemma \ref{lem:F1order}, we may assume without loss of generality that for every $P\in \Proj_1(\C^2)$ there exists a strictly increasing bijective map $f_P\colon (0,1]\to (0,1]$ such that
	\begin{equation}\label{eq:fPt}
		\phi(tP) = f_P(t) \phi(P) \quad (0<t\leq1, P\in\Proj_1(\C^2)).
	\end{equation}

	\smallskip

	\emph{STEP 2:}
	We define the following set for any qubit effect of the form $tP$, $0<t<1, P\in\Proj_1(\C^2)$:
	\begin{equation}\label{eq:EtP}
		\E_{tP} := \left\{ \frac{1}{ \tfrac{1}{1-t} \|P^\perp-Q\|^2 + \|P-Q\|^2 } Q \colon Q\in\Proj_1(\C^2)\setminus\{P, P^\perp\} \right\}.
	\end{equation}
	(For a visualisation of $\E_{tP}$ see Section \ref{sec:Visual}.)
	Using Lemma \ref{lem:P2} we see that
	\begin{equation*}
		\E_{tP} = \big((tP)^\sim \setminus \cup\{(sP)^\sim\colon t<s<1\}\big) \cap\FiniteRank_1(\C^2) \qquad (0<t<1, P\in\Proj_1(\C^2)).
	\end{equation*}
	By the properties of $\phi$ we obtain
	\begin{equation}\label{eq:phiellipsoidtP}
		\phi(\E_{tP}) = \E_{\phi(tP)} = \E_{f_P(t)\phi(P)} \qquad (0<t<1, P\in\Proj_1(\C^2)).
	\end{equation}
	
	Next, using the set introduced in \eqref{eq:EtP}, we prove the following property of $\phi$:
	\begin{equation}\label{eq:phiprojpi/4}
		\|P-Q\| = \sin \tfrac{\pi}{4} \;\iff\; \|\phi(P)-\phi(Q)\| = \sin \tfrac{\pi}{4} \qquad (P,Q\in\Proj_1(\C^2)).
	\end{equation}
	By a straightforward calculation we get that
	\begin{equation*}
		\E_{tP}\cap\E_{rP^\perp} = \left\{ \frac{1-t}{1-t \cdot s(t,r)^2}Q \colon Q\in\Proj_1(\C^2), \|P-Q\| = s(t,r) \right\} \qquad (t,r\in (0,1), P\in\Proj_1(\C^2))
	\end{equation*}
	where 
	\begin{equation*}
		s(t,r) := \sqrt{\frac{\tfrac{t}{1-t}}{\tfrac{t}{1-t}+\tfrac{r}{1-r}}}.
	\end{equation*}
	Note that $s(t,r) = \sin\tfrac{\pi}{4}$ holds if and only if $t= r$.
	By Lemma \ref{lem:P1}, this is further equivalent to the following:
	\begin{align*}
		\forall\; A_1 \in \E_{tP}\cap\E_{rP^\perp}, \; \exists\, A_2 \in \E_{tP}\cap\E_{rP^\perp}, A_1\neq A_2\colon (A_1)^\sim\cap\Proj_1(\C^2) = (A_2)^\sim\cap\Proj_1(\C^2).
	\end{align*}
	Notice that by \eqref{eq:phiellipsoidtP} this is equivalent to the following:
	\begin{align*}
		\forall\; B_1 \in \E_{f_P(t)\phi(P)}\cap\E_{f_{P^\perp}(r)\phi(P)^\perp}, \; \exists\, B_2 \in \E_{f_P(t)\phi(P)}\cap\E_{f_{P^\perp}(r)\phi(P)^\perp}, B_1\neq B_2\colon \\
		(B_1)^\sim\cap\Proj_1(\C^2) = (B_2)^\sim\cap\Proj_1(\C^2),
	\end{align*}
	which is further equivalent to $f_P(t) = f_{P^\perp}(r)$.
	
	Hence we can conclude a few important properties of $\phi$. 
	First, we have
	\begin{equation*}
		f_P(t) = f_{P^\perp}(t) \quad (0< t \leq 1, P\in\Proj_1(\C^2)).
	\end{equation*}
	Second, since for every $0<t<1$ and $P\in\Proj_1(\C^2)$ we have
	\begin{align*}
		&\left\{ f_Q\left(\tfrac{1-t}{1-t/2}\right)\phi(Q) \colon Q\in\Proj_1(\C^2), \|P-Q\| = \sin\tfrac{\pi}{4} \right\} = \phi\left(\E_{tP}\cap\E_{tP^\perp}\right) \nonumber \\
		&= \E_{f_P(t)\phi(P)}\cap\E_{f_P(t)\phi(P)^\perp} = \left\{ \tfrac{1-f_P(t)}{1-f_P(t)/2}R \colon R\in\Proj_1(\C^2), \|\phi(P)-R\| = \sin\tfrac{\pi}{4} \right\},
	\end{align*}
	therefore using \eqref{eq:fPt} gives \eqref{eq:phiprojpi/4}.
	
	Furthermore, we also obtain
	\begin{equation}\label{eq:fQ}
		f_Q\left(\tfrac{1-t}{1-t/2}\right) = \tfrac{1-f_P(t)}{1-f_P(t)/2} \qquad \left(0< t < 1, P,Q\in\Proj_1(\C^2), \|P-Q\| = \sin\tfrac{\pi}{4}\right).
	\end{equation}
	By Lemma \ref{lem:P3}, for all $Q_1, Q_2\in\Proj_1(\C^2)$ there exists a rank-one projection $P$ such that
	\begin{equation*}
		\|Q_1-P\| = \|Q_2-P\| = \sin\tfrac{\pi}{4}.
	\end{equation*}
	Therefore, applying \eqref{eq:fQ} and noticing that $t\mapsto\tfrac{1-t}{1-t/2}$ is a strictly decreasing bijection of $(0,1)$ gives that 
	\begin{equation*}
		f_{Q_1}(t) = f_{Q_2}(t) \quad (t\in (0,1), Q_1, Q_2\in\Proj_1(\C^2)).
	\end{equation*}
	Thus we conclude that there exists a strictly increasing bijection $f\colon (0,1]\to (0,1]$ such that
	\begin{equation}\label{eq:fPeqf}
		\phi(tP) = f(t) \phi(P) \quad (0<t\leq1, P\in\Proj_1(\C^2)).
	\end{equation}
	We also observe that \eqref{eq:fQ} implies
	\begin{equation}\label{eq:frule}
		f\left( \tfrac{1-t}{1-t/2} \right) = \tfrac{1-f(t)}{1-f(t)/2},
	\end{equation}
	therefore we notice that 
	\begin{equation}\label{eq:ffixpoint}
		f\left( 2-\sqrt{2} \right) = 2-\sqrt{2},
	\end{equation}
	which is a consequence of the fact that the unique solution of the equation $t = \tfrac{1-t}{1-t/2}$, $0<t<1$, is $t = 2-\sqrt{2}$.

	\smallskip
	
	\emph{STEP 3:}
	Next, applying \cite[Theorem 2.3]{Ge} gives that there exists a unitary or antiunitary operator $U\colon \C^2\to\C^2$ such that we have
	\begin{equation*}
		U^*\phi(P)U \in \{P,P^\perp\} \quad (P\in \Proj_1(\C^2)).
	\end{equation*}
	Since either both $U^*\phi(\cdot)U$ and $\phi(\cdot)$ satisfy our assumptions simultaneously, or none of them does, therefore without loss of generality we may assume that we have
	\begin{equation*}
		\phi(P) \in \{P,P^\perp\} \quad (P\in \Proj_1(\C^2)).
	\end{equation*}
	We now claim that
	\begin{equation}\label{eq:PorPperp}
		\text{either}\; \phi(P) = P \;\; (P\in \Proj_1(\C^2)), \; \text{or} \; \phi(P) = P^\perp \;\; (P\in \Proj_1(\C^2))
	\end{equation}
	Let us assume otherwise, then there exist two rank-one projections $P$ and $Q$ such that $\|P-Q\| < \sin\tfrac{\pi}{4}$, $\phi(P) = P$ and $\phi(Q) = Q^\perp$.
	Note that $\|P-Q^\perp\| = \sqrt{1-\|P-Q\|^2} > \sin\tfrac{\pi}{4} > \|P-Q\|$. 
	By \eqref{eq:phiellipsoidtP} and \eqref{eq:ffixpoint} we have
	\begin{align}\label{eq:fid1}
		\left\{\tfrac{\sqrt{2}-1}{1-(2-\sqrt{2})\|P-R\|^2} R\colon R\in\Proj_1(\C^2)\setminus\{P,P^\perp\} \right\} = \E_{(2-\sqrt{2})P} = \phi\left(\E_{(2-\sqrt{2})P}\right) \nonumber \\
		= \left\{f\left(\tfrac{\sqrt{2}-1}{1-(2-\sqrt{2})\|P-R\|^2}\right) \phi(R)\colon R\in\Proj_1(\C^2)\setminus\{P,P^\perp\} \right\}.
	\end{align}
	Therefore putting first $R=Q$ and then $R=Q^\perp$ gives
	\begin{align*}
		\phi\left( \tfrac{\sqrt{2}-1}{1-(2-\sqrt{2})\|P-Q\|^2} Q \right) & = f\left(\tfrac{\sqrt{2}-1}{1-(2-\sqrt{2})\|P-Q\|^2}\right) \phi(Q) \\
		& = f\left(\tfrac{\sqrt{2}-1}{1-(2-\sqrt{2})\|P-Q\|^2}\right) Q^\perp = \tfrac{\sqrt{2}-1}{1-(2-\sqrt{2})\|P-Q^\perp\|^2} Q^\perp
	\end{align*}
	and
	\begin{align*}
		\phi\left( \tfrac{\sqrt{2}-1}{1-(2-\sqrt{2})\|P-Q^\perp\|^2} Q^\perp \right) & = f\left(\tfrac{\sqrt{2}-1}{1-(2-\sqrt{2})\|P-Q^\perp\|^2}\right) \phi(Q^\perp) \\
		& = f\left(\tfrac{\sqrt{2}-1}{1-(2-\sqrt{2})\|P-Q^\perp\|^2}\right) Q = \tfrac{\sqrt{2}-1}{1-(2-\sqrt{2})\|P-Q\|^2} Q.
	\end{align*}
	But this implies that $f$ interchanges two different numbers which contradicts to its strict increasingness -- proving our claim \eqref{eq:PorPperp}.

	Note that for every $0\leq \vartheta \leq \tfrac{\pi}{2}$ and $0\leq \mu < 2\pi$ we have
	\begin{equation*}
		(P_{(\cos\vartheta, e^{i\mu} \sin\vartheta)})^\perp
		= \left[ \begin{matrix}
			\sin^2\vartheta & -e^{-i\mu} \cos\vartheta \sin\vartheta \\
			-e^{i\mu} \cos\vartheta \sin\vartheta & \cos^2\vartheta \\
		\end{matrix}\right]
		= \left[ \begin{matrix}
			0 & 1 \\
			-1 & 0 \\
		\end{matrix}\right]
		(P_{(\cos\vartheta, e^{i\mu} \sin\vartheta)})^t 
		\left[ \begin{matrix}
			0 & 1 \\
			-1 & 0 \\
		\end{matrix}\right]^*
	\end{equation*}
	where $\cdot^t$ stands for the transposition, and we used the notation of the Busch--Gudder theorem.
	It is well-known, and can be verified by an easy computation, that we have $A^t = KAK^*$ for every qubit effect $A$, where $K$ is the coordinate-wise conjugation antiunitary operator: $K(z_1,z_2) = (\overline{z_1},\overline{z_2})$ $(z_1,z_2\in\C)$.
	Therefore from now on we may assume without loss of generality that we have
	\begin{equation}\label{eq:fixedP}
		\phi(P) = P \quad (P\in\Proj_1(\C^2)),
	\end{equation}
	i.e.~$\phi$ fixes all rank-one projections.
	
	Finally, observe that \eqref{eq:fid1} and \eqref{eq:fixedP} implies
	\begin{equation*}
		f\left(\tfrac{\sqrt{2}-1}{1-(2-\sqrt{2}) \tau }\right) = \tfrac{\sqrt{2}-1}{1-(2-\sqrt{2}) \tau} \quad (0<\tau<1),
	\end{equation*}
	thus we obtain $\phi(tP) = tP$ for all $P\in\Proj_1(\C^2)$ and $\sqrt{2}-1 < t < 1$. 
	But this further implies 
	\begin{align*}
		\left\{\tfrac{1-t}{1-t\|P-Q\|^2} Q\colon Q\in\Proj_1(\C^2)\setminus\{P,P^\perp\} \right\} = \E_{tP} = \phi\left(\E_{tP}\right)
		= \left\{ f\left(\tfrac{1-t}{1-t\|P-Q\|^2}\right) Q\colon Q\in\Proj_1(\C^2)\setminus\{P,P^\perp\} \right\}
	\end{align*}
	for all $\sqrt{2}-1 < t < 1$, from which we conclude  
	\begin{equation}\label{eq:F1fix}
		\phi(tP) = tP \quad (0<t<1),
	\end{equation}
	i.e.~$\phi$ fixes all rank-one effects.
	From here we only need to apply Corollary \ref{cor:CoexSet2Dim} and transform back to our original $\phi$ to complete the proof.
\end{proof}

%-------------------------------------------------------------------------------------------------------
%-------------------------------------------------------------------------------------------------------
%-------------------------------------------------------------------------------------------------------
%-------------------------------------------------------------------------------------------------------%-------------------------------------------------------------------------------------------------------

\section{Proof of Theorem \ref{thm:main} in the general case}

Here we prove the general case of our main theorem, utilising the above proved low-dimensional case.
We start with two lemmas.

\begin{lemma}\label{lem:ProjChar0}
	Let $P \in \Proj(H)\setminus \Sca(H)$ and $A \in \Effect(H) \setminus \{ P, P^\perp \}$.
	Then there exists a rank-one effect $R\in \FiniteRank_1(H)$ such that $R \sim A$ but $R \not\sim P$.
\end{lemma}

\begin{proof}
	Assume that $A \in \Effect(H)$ such that $A^\sim \cap \FiniteRank_1(H) \subseteq P^\sim = P^c$ holds. 
	We have to show that then either $A=P$, or $A=P^\perp$.
	Clearly, $A$ is not a scalar effect.
	By Corollary \ref{cor:RankOneStrengthFunct} we obtain that 
	\begin{equation*}
		\Lambda(A,Q) + \Lambda(A^\perp,Q) \leq \Lambda(P,Q) + \Lambda(P^\perp,Q) \qquad (Q\in\Proj_1(H)).
	\end{equation*}
	Notice that the set
	\begin{align*}
		\supp \left(\Lambda(P,\cdot) + \Lambda(P^\perp,\cdot)\right) := \left\{ Q\in\Proj_1(H) \colon \Lambda(P,Q) + \Lambda(P^\perp,Q) > 0 \right\}
	\end{align*}
	has two connected components (with respect to the operator norm topology), namely
	\begin{align}\label{eq:ConnComp}
		\left\{ Q\in\Proj_1(H) \colon \Image Q \subset \Image P \right\} \;\; \text{and} \;\; \left\{ Q\in\Proj_1(H) \colon \Image Q \subset \Ker P \right\}.
	\end{align}
	However, by the Busch--Gudder theorem we obtain that
	\begin{align*}
		\left\{ Q\in\Proj_1(H) \colon \Image Q \subset \Image A \cup \Image (I-A) \right\}
		& \subseteq \left\{ Q\in\Proj_1(H) \colon \Image Q \subset \Image A^{1/2} \cup \Image (I-A)^{1/2} \right\} \\
		& \subseteq \supp \left(\Lambda(A,\cdot) + \Lambda(A^\perp,\cdot)\right) \subseteq \supp \left(\Lambda(P,\cdot) + \Lambda(P^\perp,\cdot)\right).
	\end{align*}
	Since $\supp \left(\Lambda(P,\cdot) + \Lambda(P^\perp,\cdot)\right)$ is a closed set, we obtain
	\begin{align}\label{eq:ClosedRel}
		\left\{ Q\in\Proj_1(H) \colon \Image Q \subset (\Image A)^- \cup (\Image (I-A))^- \right\}
		& = \left\{ Q\in\Proj_1(H) \colon \Image Q \subset (\Ker A)^\perp \cup (\Ker (I-A))^\perp \right\} \nonumber \\
		& \subseteq \supp \left(\Lambda(P,\cdot) + \Lambda(P^\perp,\cdot)\right).
	\end{align}
	Notice that the left-hand side of \eqref{eq:ClosedRel} is connected if and only if $A$ is not a projection, in which case it must be a subset of one of the components of the right-hand side.
	However, this is impossible because the left-hand side contains a maximal set of pairwise orthogonal rank-one projections.
	Therefore $A\in\Proj(H)$, and in particular $\supp \left(\Lambda(A,\cdot) + \Lambda(A^\perp,\cdot)\right)$ has two connected components.
	From here using \eqref{eq:ConnComp} for both $A$ and $P$ we easily complete the proof.	
\end{proof}

We introduce a new relation on $\Effect(H)\setminus \Sca(H)$. 
For $A,B \in \Effect(H)\setminus \Sca(H)$ we write $A \prec B$ if and only if for every $C\in A^\sim \setminus \Sca(H)$ there exists a $D\in B^\sim \setminus \Sca(H)$ such that $C^\sim \subseteq D^\sim$.
Clearly, for every non-scalar effect $B$ we have $B \prec B$ and $B^\perp \prec B$. 
In particular $\prec$ is a reflexive relation, but it is not antisymmetric.
It is also straightforward from the definition that $\prec$ is a transitive relation, i.e. $A \prec B$ and $B \prec C$ imply $A \prec C$.

We proceed with characterising non-trivial projections in terms of the relation of coexistence.

\begin{lemma}\label{lem:ProjChar}
Assume that $A \in \Effect(H)\setminus \Sca(H)$. Then the following two statements are equivalent:
	\begin{itemize}
		\item[\textup{(i)}] $A \in \Proj (H)$,
		\item[\textup{(ii)}] $\# \{ B \in \Effect (H)\setminus \Sca(H) \colon B \prec A \} = 2$.
	\end{itemize}
\end{lemma}

\begin{proof}
\emph{(i)$\Longrightarrow$(ii):}
Suppose that $B \in \Effect(H)\setminus \Sca(H)$, $B\neq A$, $B\neq A^\perp$ and $B \prec A$. 
We need to show that this assumption leads to a contradiction.
By Lemma \ref{lem:ProjChar0} there exists a rank one effect $tQ$, with some $Q\in\Proj_1(H)$ and $t \in (0,1]$, such that $tQ \sim B$ but $tQ\not\sim A$. 
From $B \prec A$ we know that there exists a non-scalar effect $D$ such that
\begin{equation*}
	(tQ)^\sim \subseteq D^\sim \ \ \ {\rm and} \ \ \ D  \sim A.
\end{equation*}
By Lemma \ref{lem:CoexSubset} (a) we have 
$$
D \in (tQ)''\cap \Effect(H) = Q''\cap \Effect(H) = \left\{ sQ + r Q^\perp \in \Effect(H) \colon s,r \in [0,1] \right\},
$$
where the latter equation is easy to see (even in non-separable Hilbert spaces). 
Since we also have $D\in A^c$, we obtain $Q\in A^c$, hence the contradiction $tQ\in A^c = A^\sim$.

\emph{(ii)$\Longrightarrow$(i):}
Here we use contraposition, so let us assume that $A \in \left(\Effect(H)\setminus\Proj(H)\right)\setminus \Sca(H)$. 
We shall construct a non-trivial projection $P$ (which is obviously different from both $A$ and $A^\perp$) such that $P\prec A$.
First, notice that there exists an $0 < \varepsilon < \tfrac{1}{2}$ such that $H_A\left( \left( \varepsilon, 1-\varepsilon \right] \right) \notin \left\{ \{0\}, H \right\}$.
Indeed, otherwise an elementary examination of the spectrum gives that $\sigma(A) \subseteq \{\varepsilon_0,1-\varepsilon_0\}$ holds with some $0 < \varepsilon_0 < \tfrac{1}{2}$. 
As $A$ is non-scalar, we actually get $\sigma(A) = \{\varepsilon_0,1-\varepsilon_0\}$, which implies that $H_A\left( \left( \varepsilon_0, 1-\varepsilon_0 \right] \right)$ is a non-trivial subspace.

Let us now consider the orthogonal decomposition $H = H_1 \oplus H_2 \oplus H_3$ where
\begin{equation*}
	H_1 = H_A\left( \left[ 0, \varepsilon \right] \right), \;\;
	H_2 = H_A\left( \left( \varepsilon, 1-\varepsilon \right] \right) \;\; \text{and} \;\;
	H_3 = H_A\left( \left( 1-\varepsilon, 1 \right] \right).
\end{equation*}
With respect to this orthogonal decomposition we have
\begin{equation*}
	A =  
	\left[ \begin{matrix} 
		A_1 & 0 & 0 \\
		0 & A_2 & 0 \\ 
		0 & 0 & A_3 
		\end{matrix} \right] \in \Effect(H_1 \oplus H_2 \oplus H_3).
\end{equation*}
Since coexistence is invariant under taking the ortho-complements, we may assume without loss of generality that $H_3\neq\{0\}$. 
Let us set 
\begin{equation*}
	P = \left[ \begin{matrix} 
			I & 0 & 0 \\
			0 & I & 0 \\ 
			0 & 0 & 0 
		\end{matrix} \right] \notin \Sca(H_1 \oplus H_2 \oplus H_3).
\end{equation*}
Our goal is to show that $P \prec A$. 
Let $C$ be an arbitrary non-scalar effect coexistent with $P$. Then, since $C$ and $P$ commute, the matrix form of $C$ is
\begin{equation*}
	C = \left[ \begin{matrix} 
		C_{11} & C_{12} & 0 \\
		C_{12}^* & C_{22} & 0 \\ 
		0 & 0 & C_{33} 
		\end{matrix} \right] \in \Effect(H_1 \oplus H_2 \oplus H_3).
\end{equation*}
Consider the effect $D := \varepsilon\cdot C$ and notice that
\begin{equation*}
	\varepsilon \cdot \left[ \begin{matrix} 
		C_{11} & C_{12} & 0 \\
		C_{12}^* & C_{22} & 0 \\ 
		0 & 0 & 0 
		\end{matrix} \right] \leq I-A \;\; \text{and} \;\; 
	\varepsilon \cdot \left[ \begin{matrix} 
		0 & 0 & 0 \\
		0 & 0 & 0 \\ 
		0 & 0 & C_{33} 
		\end{matrix} \right] \leq A.
\end{equation*}
Clearly, by Lemmas \ref{lem:MANIA} and \ref{lem:CoexSubset} we have $D\sim A$ and $C^\sim \subseteq D^\sim$, which completes the proof.
\end{proof}

Next, we characterise commutativity preservers on $\Proj(H)$.
We note that the following theorem has been proved before implicitly in \cite{MoS1} for separable spaces, and was stated explicitly in \cite[Theorem 2.8]{MoS2}. 
In order to prove the theorem for general spaces, one only has to use the ideas of \cite{MoS1}, however, we decided to include the proof for the sake of completeness and clarity.

\begin{theorem}\label{thm:ProjComm}
Let $H$ be a Hilbert space of dimension at least three and $\phi\colon \Proj(H)\to \Proj(H)$ be a bijective mapping that preserves commutativity in both directions, i.e.
\begin{equation}\label{eq:ProjComm}
	PQ = QP \;\; \iff \;\; \phi(P)\phi(Q) = \phi(Q)\phi(P) \qquad (P,Q \in \Proj(H)).
\end{equation}
Then there exists a unitary or antiunitary operator $U\colon H\to H$ such that
\begin{equation*}
	\phi(P) \in \{ UPU^*, UP^\perp U^* \} \qquad (P \in \Proj(H)).
\end{equation*}
\end{theorem}

\begin{proof}
	For an arbitrary set $\calM \subseteq \Proj(H)$ let us use the following notations: $\calM^\pc := \calM^c \cap \Proj(H)$ and $\calM^{\pc\pc} := (\calM^\pc)^\pc$.
	By the properties of $\phi$ we immediately get $\phi(\calM^\pc) = \phi(\calM)^\pc$ and $\phi(\calM^{\pc\pc}) = \phi(\calM)^{\pc\pc}$ for all subset $\calM$.
	
	Next, let $P$ and $Q$ be two arbitrary commuting projections.
	Then (for instance by the Halmos's two projections theorem) we have
	\begin{align*}
		P = 
		\left[\begin{matrix}
			I & 0 & 0 & 0 \\
			0 & I & 0 & 0 \\
			0 & 0 & 0 & 0 \\
			0 & 0 & 0 & 0 \\
		\end{matrix}\right] \;\; \text{and} \;\;
		Q = 
		\left[\begin{matrix}
			I & 0 & 0 & 0 \\
			0 & 0 & 0 & 0 \\
			0 & 0 & I & 0 \\
			0 & 0 & 0 & 0 \\
		\end{matrix}\right] \in \Bdd(H_1\oplus H_2 \oplus H_3 \oplus H_4)
	\end{align*}
	where $H_1 = \Image P \cap \Image Q$, $H_2 = \Image P \cap \Ker Q$, $H_3 = \Ker P \cap \Image Q$, $H_4 = \Ker P \cap \Ker Q$ and $H = H_1\oplus H_2 \oplus H_3 \oplus H_4$.
	Note that some of these subspaces might be trivial.
	We observe that
	\begin{align*}
		\{P,Q\}^{\pc\pc} = (\{P,Q\}^\pc)^\pc 
		&= \left\{ 
			\left[\begin{matrix}
				R_1 & 0 & 0 & 0 \\
				0 & R_2 & 0 & 0 \\
				0 & 0 & R_3 & 0 \\
				0 & 0 & 0 & R_4 \\
			\end{matrix}\right] \colon R_j \in \Proj(H_j),\; j=1,2,3,4
		 \right\}^\pc \\
		 &= \left\{ 
			\left[\begin{matrix}
				\lambda_1 I & 0 & 0 & 0 \\
				0 & \lambda_2 I & 0 & 0 \\
				0 & 0 & \lambda_3 I & 0 \\
				0 & 0 & 0 & \lambda_4 I \\
			\end{matrix}\right] \colon \lambda_j \in \{0,1\},\; j=1,2,3,4
		 \right\}.
	\end{align*}
	Hence we conclude that $\#\{P,Q\}^{\pc\pc} = 2^{\#\{j\colon H_j \neq \{0\}\}}$.
	In particular, $\#\{P,Q\}^{\pc\pc} = 2$ if and only if $P,Q\in\{0,I\}$, and $\#\{P,Q\}^{\pc\pc} = 4$ if and only if either $P\notin\{0,I\}$ and $Q\in\{I,0,P,P^\perp\}$, or $Q\notin\{0,I\}$ and $P\in\{I,0,Q,Q^\perp\}$.
	
	Now, we easily conclude the following characterisation of rank-one and co-rank-one projections:
	\begin{equation*}
		P \;\text{or}\; P^\perp \in \Proj_1(H) \;\;\iff\;\; \#\{P,Q\}^{\pc\pc} \in \{4,8\} \;\;\text{holds for all}\; Q\in P^\pc.
	\end{equation*}
	This implies that
	\begin{equation*}
		\phi(\{P \colon P \;\text{or}\; P^\perp \in \Proj_1(H)\}) = \{P \colon P \;\text{or}\; P^\perp \in \Proj_1(H)\}.
	\end{equation*}
	Note that we also have $\phi(P^\perp) = \phi(P)^\perp$ for every $P\in \Proj(H)$, as $P^\pc = Q^\pc$ holds exactly when $P=Q$ or $P+Q = I$.
	Since changing the images of some pairs of ortho-complemented projections to their orto-complementations does not change the property \eqref{eq:ProjComm}, we may assume without loss of generality that $\phi(\Proj_1(H)) = \Proj_1(H)$.
	It is easy to see that two rank-one projections commute if and only if either they coincide, or they are orthogonal to each other.
	Thus, as $\dim H \geq 3$, Uhlhorn's theorem \cite{U} gives that there exist a unitary or antiunitary operator $U\colon H\to H$ such that 
	\begin{equation*}
		\phi(P) = UPU^* \qquad (P \in \Proj_1(H)).
	\end{equation*}
	
	Finally, note that for every projection $Q\in\Proj(H)$ we have
	\begin{equation*}
		Q^\pc\cap\Proj_1(H) = \{ P\in\Proj_1(H) \colon \Image P \subset \Image Q \cup \Ker Q \},
	\end{equation*}
	from which we easily complete the proof.
\end{proof}

Before we prove Theorem \ref{thm:main} in the general case, we need one more technical lemma for non-separable Hilbert spaces.
We will use the notation $\Effect_{fs}(H)$ for the set of all effects whose spectrum has finitely many elements.

\begin{lemma}\label{lem:ccFinSpec}
	For all $A\in \Effect_{fs}(H)$ we have 
	\begin{equation*}
		A^{cc} = A'' \cap \Effect(H) = \{ p(A)\in \Effect(H) \colon p \text{ is a polynomial} \}.
	\end{equation*}
\end{lemma}

\begin{proof}
	We only have to observe the following for all $A\in \Effect(H)$ with $\#\sigma(A) = n \in\N$, where $E_1, \dots E_n$ are the spectral projections and $H_j = \Image E_j$ $(j=1,2,\dots n)$:
	\begin{align*}
		A^{cc} & = \left( \bigcap_{j=1}^n E_j^c \right)^c = \left\{ \bigoplus_{j=1}^n B_j\colon B_j \in \Effect\left(H_j\right) \text{ for all } j \right\}^c = \left\{ \sum_{j=1}^n \mu_j E_j \colon \mu_j \in [0,1] \text{ for all } j \right\} \\
		&= \left\{ \bigoplus_{j=1}^n T_j \colon T_j \in \Bdd(H_j) \text{ for all } j \right\}' \cap \Effect(H) = \left( \bigcap_{j=1}^n E_j' \right)' \cap \Effect(H) = A'' \cap \Effect(H).
	\end{align*}
\end{proof}

Now, we are in the position to prove our second main theorem in the general case.

\begin{proof}[Proof of Theorem \ref{thm:main} for spaces of dimension at least three]
The proof will be divided into the following steps:
\begin{itemize}
	\item[1] we show that $\phi$ maps $\Effect_{fs}(H)$ onto itself,
	\item[2] we prove that $\phi$ has the form \eqref{eq:nonScaPhi} on $\Effect_{fs}(H)\setminus \Sca(H)$,
	\item[3] we show that $\phi$ has the form \eqref{eq:nonScaPhi} on $\Effect(H)\setminus \Sca(H)$.
\end{itemize} 

\smallskip

\emph{STEP 1:} 
First, similarly as in the previous section, we easily get the existence of a bijective function $g\colon [0,1] \to [0,1]$ such that
\begin{equation*}
	\phi(t I) = g(t)I \qquad (t\in[0,1]).
\end{equation*}
Of course, the properties of $\phi$ imply $\phi(A)^\sim = \phi(A^\sim)$ for all $A\in \Effect(H)$, and also
\begin{equation}\label{eq:philesssim}
	B^\sim \subseteq A^\sim \;\;\iff\;\; \phi(B)^\sim \subseteq \phi(A)^\sim \qquad (A,B\in \Effect(H)).
\end{equation}
From the latter it follows that
\begin{equation}\label{eq:philesssim}
	B \prec A \;\;\iff\;\; \phi(B) \prec \phi(A) \qquad (A,B\in \Effect(H)\setminus \Sca(H)).
\end{equation}
Hence by Lemma \ref{lem:ProjChar} we obtain 
\begin{equation*}
	\phi(P(H)\setminus\{0,I\}) = P(H)\setminus\{0,I\}, 
\end{equation*}
and therefore Lemma \ref{lem:properties} (b) implies that the restriction $\phi|_{P(H)\setminus\{0,I\}}$ preserves commutativity in both directions. 
Applying Theorem \ref{thm:ProjComm} then gives that up to unitary--antiunitary equivalence and element-wise ortho-complementation, we have
\begin{equation}\label{eq:ProjFix}
	\phi(P) = P \qquad (P\in P(H)\setminus\{0,I\}).
\end{equation}
From now on we may assume without loss of generality that this is the case.

Next, by the spectral theorem \cite[Theorem IX.2.2]{Conway} we have
\begin{equation*}
	A^c = \bigcap_{\Delta\in\mathcal{B}_{[0,1]}} E_A(\Delta)^c = \bigcap_{\Delta\in\mathcal{B}_{[0,1]}} E_A(\Delta)^\sim \qquad (A\in \Effect(H)).
\end{equation*}
Therefore we obtain
\begin{align*}
	\phi(A^c) 
	= \bigcap_{\Delta\in\mathcal{B}_{[0,1]}} \phi(E_A(\Delta))^\sim 
	= \bigcap_{\Delta\in\mathcal{B}_{[0,1]}} E_A(\Delta)^\sim = A^c \quad (A\in \Effect(H)),
\end{align*}
and thus also
\begin{equation*}
	\phi(A^{cc}) = \phi\left(\bigcap_{B\in A^c} B^c\right) = \bigcap_{B\in A^c} \phi\left(B^c\right) = \bigcap_{B\in A^c} B^c = A^{cc} \qquad (A\in \Effect(H)).
\end{equation*}
In particular, we have
\begin{equation*}
	\phi(A)\in A^{cc} \qquad (A\in \Effect(H)).
\end{equation*} 
Hence for all $A\in \Effect_{fs}(H)$ there exists a polynomial $p_A$ such that $p_A(\sigma(A)) \subset [0,1]$ and
\begin{align*}
	\phi(A) = p_A(A) \qquad (A\in \Effect_{fs}(H)).
\end{align*}
As a similar statement holds for $\phi^{-1}$, we immediately get $\phi(\Effect_{fs}(H)) = \Effect_{fs}(H)$.
Also, notice that $\#\sigma(\phi(A)) = \#\sigma(p_A(A)) \leq \#\sigma(A)$ and $\#\sigma(\phi^{-1}(A)) \leq \#\sigma(A)$ hold for all $A\in \Effect_{fs}(H)$.
Whence we obtain 
\begin{equation}\label{eq:SpecCard}
	\#\sigma(\phi(A))= \#\sigma(A) \qquad (A\in \Effect_{fs}(H)).
\end{equation}
In particular, the restriction $p_A|_{\sigma(A)}$ is injective.

\smallskip

\emph{STEP 2:}
Now, let $M$ be an arbitrary two-dimensional subspace of $H$ and let $P_M\in\Proj(H)$ be the orthogonal projection onto $M$. 
Consider two arbitrary effects $A,B\in (P_M)^\sim \cap \Effect_{fs}(H)$ which therefore have the following matrix representations:
\begin{equation*}
	A = \left[\begin{matrix}
		A_M & 0 \\
		0 & A_{M^\perp}
	\end{matrix}\right]
	\quad\text{and}\quad
	B = \left[\begin{matrix}
		B_M & 0 \\
		0 & B_{M^\perp}
	\end{matrix}\right] \in \Effect_{fs}(M \oplus M^\perp).
\end{equation*}
Obviously,
\begin{equation*}
	\phi(A) = p_A(A) = \left[\begin{matrix}
		p_A(A_M) & 0 \\
		0 & p_A(A_{M^\perp})
	\end{matrix}\right]
	\quad\text{and}\quad
	\phi(B) = p_B(B) = \left[\begin{matrix}
		p_B(B_M) & 0 \\
		0 & p_B(B_{M^\perp})
	\end{matrix}\right].
\end{equation*}
Note that by \eqref{eq:SpecCard}, the polynomial $p_A$ acts injectively on $\sigma(A)$, therefore 
\begin{equation*}
	A_M\in \Sca(M) \;\;\iff\;\; p_A(A_M)\in \Sca(M),
\end{equation*}
and of course, similarly for $B$.
We observe that by Lemma \ref{lem:dirsum} the following two equations hold:
\begin{align}\label{eq:2dim1}
	A^\sim
	\bigcap
	\left[\begin{matrix}
		I & 0 \\
		0 & 0
	\end{matrix}\right]^\sim
	\bigcap
	\left(
		\bigcap_{P\in\Proj_1(M^\perp)}
		\left[\begin{matrix}
			0 & 0 \\
			0 & P
		\end{matrix}\right]^\sim
	\right)	
	=
	\left\{
		\left[\begin{matrix}
			D & 0 \\
			0 & \lambda I
		\end{matrix}\right] 
		\colon D\sim A_M, \lambda\in[0,1]
	\right\}
\end{align}
and
\begin{align}\label{eq:2dim2}
	\phi(A)^\sim
	\bigcap
	\left[\begin{matrix}
		I & 0 \\
		0 & 0
	\end{matrix}\right]^\sim
	\bigcap
	\left(
		\bigcap_{P\in\Proj_1(M^\perp)}
		\left[\begin{matrix}
			0 & 0 \\
			0 & P
		\end{matrix}\right]^\sim
	\right)	
	=
	\left\{
		\left[\begin{matrix}
			D & 0 \\
			0 & \lambda I
		\end{matrix}\right] 
		\colon D\sim p_A(A_M), \lambda\in[0,1]
	\right\}.
\end{align}
It is important to observe that by \eqref{eq:ProjFix} the set in \eqref{eq:2dim2} is the $\phi$-image of \eqref{eq:2dim1}.
Thus we obtain the following equivalence if $A_M \notin \Sca(M)$:
\begin{align}\label{eq:2x2cornercoex}
	B_M \in \{A_M, A_M^\perp\}
	\;&\iff\;
	A_M^\sim = B_M^\sim \nonumber
\\
	&\iff\;
	\left\{
		\left[\begin{matrix}
			D & 0 \\
			0 & \lambda I
		\end{matrix}\right] 
	\colon D\sim A_M, \lambda\in[0,1]\right\}
	=
	\left\{
		\left[\begin{matrix}
			E & 0 \\
			0 & \mu I
		\end{matrix}\right] 
	\colon E\sim B_M, \mu\in[0,1]\right\} \nonumber
\\
	&\iff\;
	\left\{
		\left[\begin{matrix}
			D & 0 \\
			0 & \lambda I
		\end{matrix}\right] 
	\colon D\sim p_A(A_M), \lambda\in[0,1]\right\}
	=
	\left\{
		\left[\begin{matrix}
			E & 0 \\
			0 & \mu I
		\end{matrix}\right] 
	\colon E\sim p_B(B_M), \mu\in[0,1]\right\} \nonumber
\\
	&\iff\;
	\left(p_A(A_M)\right)^\sim = \left(p_B(B_M)\right)^\sim \nonumber
\\
	&\iff\;
	p_B(B_M) \in \left\{p_A(A_M), I-p_A(A_M) \right\}.
\end{align}

Now, we are in the position to use the previously proved two-dimensional version.
Let
\begin{equation*}
	\mathfrak{E}(M) := \left\{ \left\{D,D^\perp\right\} \colon D\in \Effect(M)\setminus \Sca(M) \right\} \cup \{\Sca(M)\},
\end{equation*}
and let us say that two elements of $\mathfrak{E}(M)$ are coexistent, in notation $\approx$, if either one of them is $\Sca(M)$, or the two elements are $\{D,D^\perp\}$ and $\{E,E^\perp\}$ with $D\sim E$.
Clearly, the bijective restriction 
\begin{equation*}
	\phi|_{(P_M)^\sim \cap \Effect_{fs}(H)}\colon (P_M)^\sim \cap \Effect_{fs}(H) \to (P_M)^\sim \cap \Effect_{fs}(H)
\end{equation*}
induces a well-defined bijection on $\mathfrak{E}(M)$ by 
\begin{equation*}
	\Sca(M)\mapsto \Sca(M), \; \{A_M,A_M^\perp\}\mapsto \{p_A(A_M), p_A(A_M)^\perp\} \qquad (A_M \notin \Sca(M)).
\end{equation*}
Notice that this map also preserves the relation $\approx$ in both directions.
Indeed, for all $A,B\in (P_M)^\sim \cap \Effect_{fs}(H)$, $A_M,B_M \notin \Sca(H)$ we have
\begin{align*}
	\{A_M, A_M^\perp\} \approx \{B_M, B_M^\perp\} 
	&\;\iff\; \hat{A} := \left[\begin{matrix}
			A_M & 0 \\
			0 & 0
		\end{matrix}\right] \sim 
		\hat{B} := \left[\begin{matrix}
			B_M & 0 \\
			0 & 0
		\end{matrix}\right] \\
	&\;\iff\; \left[\begin{matrix}
			p_{\hat{A}}(A_M) & 0 \\
			0 & p_{\hat{A}}(0) I
		\end{matrix}\right] \sim 
		\left[\begin{matrix}
			p_{\hat{B}}(B_M) & 0 \\
			0 & p_{\hat{B}}(0) I
		\end{matrix}\right]\\
	&\;\iff\; \{p_{\hat{A}}(A_M), p_{\hat{A}}(A_M)^\perp\} \approx \{p_{\hat{B}}(B_M), p_{\hat{B}}(B_M)^\perp\} \\
	&\;\iff\; \{p_{A}(A_M), p_{A}(A_M)^\perp\} \approx \{p_{B}(B_M), p_{B}(B_M)^\perp\}.
\end{align*}
Therefore, using the two-dimensional version of Theorem \ref{thm:main}, we obtain a unitary or antiunitary operator $U_M\colon M\to M$ such that 
\begin{equation*}
	p_A(A_M) \in \{U_M (A_M) U_M^*, U_M (A_M)^\perp U_M^*\} \qquad (A\in (P_M)^\sim \cap \Effect_{fs}(H), \; A_M\notin \Sca(M))
\end{equation*}
and
\begin{equation*}
	p_A(A_M) \in \Sca(M) \qquad (A\in (P_M)^\sim \cap \Effect_{fs}(H), \; A_M\in \Sca(M)).
\end{equation*}

Observe that this implies the following: for any pair of orthogonal unit vectors $x,y\in M$ we must have either $U_M(\C\cdot x) = \C\cdot x$ and $U_M(\C\cdot y) = \C\cdot y$, or $U_M(\C\cdot x) = \C\cdot y$ and $U_M(\C\cdot y) = \C\cdot x$.
As $U_M$ is continuous, we have either the first case for all orthogonal pairs $\C\cdot x,\C\cdot y$, or the second for every such pair.
But a similar statement holds for all two-dimensional subspaces, therefore it is easy to show that the second possibility cannot occur. 
Consequently, we have $U_M(\C\cdot x) = \C\cdot x$ for all unit vectors $x\in M$, from which it follows that $U_M$ is a scalar multiple of the identity operator.
Thus we obtain the following for every two-dimensional subspace $M$:
\begin{equation*}
	p_A(A_M) \in \{A_M, A_M^\perp\} \qquad (A\in (P_M)^\sim \cap \Effect_{fs}(H), \; A_M\notin \Sca(M)).
\end{equation*} 
From here it is rather straightforward to obtain 
\begin{equation}\label{eq:EfsPhi}
	\phi(A) = p_A(A) \in \{A, A^\perp\} \qquad (A\in \Effect_{fs}(H)\setminus \Sca(M)).
\end{equation} 

\smallskip

\emph{STEP 3:}
Observe that \eqref{eq:EfsPhi} holds for every $A\in \FiniteRank(H)$, therefore an application of Theorem \ref{thm:CoexSet} and Corollary \ref{cor:PerpCoex} completes the proof in the separable case.
As for the general case, let us consider an arbitrary effect $A\in \Effect(H)\setminus \Effect_{fs}(H)$ and an orthogonal decomposition $H = \oplus_{i\in\mathcal{I}} H_i$ such that each $H_i$ is a separable invariant subspace of $A$. 
By \eqref{eq:ProjFix} and Lemma \ref{lem:properties} (b), each $H_i$ is an invariant subspace also for $\phi(A)$, in particular, we have
\begin{equation*}
	A = \oplus_{i\in\mathcal{I}} A_i, \;\;\text{and}\;\; \phi(A) = \oplus_{i\in\mathcal{I}} \mathcal{A}_i \in \Effect(\oplus_{i\in\mathcal{I}} H_i).
\end{equation*}
Without loss of generality we may assume from now on that there exists an $i_0\in\mathcal{I}$ so that $A_{i_0}$ is not a scalar effect.

Now, let $i\in\mathcal{I}$, $F\in\FiniteRank(H)$ and $\Image F \subseteq H_{i}$ be arbitrary.
Then by \eqref{eq:EfsPhi} we have
\begin{equation*}
	A_{i} \sim P_{i}F|_{H_{i}} \;\;\iff\;\; A\sim F \;\;\iff\;\; \phi(A)\sim F \;\;\iff\;\; \mathcal{A}_{i} \sim P_{i}F|_{H_{i}}.
\end{equation*}
In particular, $A_{i}^\sim\cap\FiniteRank(H_{i}) = \mathcal{A}_{i}^\sim\cap\FiniteRank(H_{i})$, therefore by Theorem \ref{thm:CoexSet} we get that for all $i$ we have either $A_i, \mathcal{A}_i\in \Sca(H)$, or $A_i = \mathcal{A}_i$, or $\mathcal{A}_i = A_i^\perp$.
By considering $A^\perp$ instead of $A$ if necessary, we may assume that we have $A_{i_0} = \mathcal{A}_{i_0}$.
Finally, for any $i_1\in\mathcal{I}\setminus\{i_0\}$ let us consider the orthogonal decomposition $H = \oplus_{i\in\mathcal{I}\setminus\{i_0,i_1\}} H_i \oplus (H_{i_0} \oplus H_{i_1}) $.
Similarly as above, we then get $A_{i_0}\oplus A_{i_1} = \mathcal{A}_{i_0} \oplus \mathcal{A}_{i_1}$, and the proof is complete.
\end{proof}

%-------------------------------------------------------------------------------------------------------

\section{A remark on the qubit case}\label{sec:Visual}

Here we visualise the set $A^\sim\cap\FiniteRank_1(\C^2)$ for a general rank-one qubit effect $A$.
	First, let us introduce Bloch's representation.
	Consider the following vector space isomorphism between the space of all 2$\times$2 Hermitian matrices $\Bdd_{sa}(\C^2)$ and $\R^4$, see also \cite{BuS}:
	\begin{equation*}
		\rho\colon \Bdd_{sa}(\C^2) \to \R^4, \quad \rho(A) = \rho(x_0\sigma_0+x_1\sigma_1+x_2\sigma_2+x_3\sigma_3) = (x_0,x_1,x_2,x_3),
	\end{equation*}
	where
	\begin{equation*}
		\sigma_0 = \left[\begin{matrix}
			1 & 0 \\
			0 & 1
		\end{matrix}\right], \;
		\sigma_1 = \left[\begin{matrix}
			0 & 1 \\
			1 & 0
		\end{matrix}\right], \;
		\sigma_2 = \left[\begin{matrix}
			0 & -i \\
			i & 0
		\end{matrix}\right], \;
		\sigma_3 = \left[\begin{matrix}
			1 & 0 \\
			0 & -1
		\end{matrix}\right]
	\end{equation*}
	are the Pauli matrices.
	Clearly, we have $\rho(0) = (0,0,0,0)$, $\rho(I) = (1,0,0,0)$.
	The Bloch representation is usually defined as the restriction $\rho|_{\Proj_1(\C^2)}$ which maps $\Proj_1(\C^2)$ onto a sphere of the three-dimensional affine subspace $\{ (1/2,x_1,x_2,x_3) \colon x_j\in\R, j=1,2,3 \}$ with centre at $(1/2,0,0,0)$ and radius $1/2$.
	Indeed, as the general form of a rank-one projection in $\C^2$ is
	\begin{equation*}
		P_{(\cos\vartheta, e^{i\mu} \sin\vartheta)} =
		\left[ \begin{matrix}
			\cos\vartheta \\
			e^{i\mu} \sin\vartheta
		\end{matrix}\right]
		\left[ \begin{matrix}
		\cos\vartheta \\
		e^{i\mu} \sin\vartheta
		\end{matrix}\right]^*
		=
		\left[ \begin{matrix}
		\cos^2\vartheta & e^{-i\mu} \cos\vartheta \sin\vartheta \\
		e^{i\mu} \cos\vartheta \sin\vartheta & \sin^2\vartheta \\
		\end{matrix}\right]
	\end{equation*}
	where $0\leq \vartheta \leq \tfrac{\pi}{2}$ and $0\leq \mu < 2\pi$, a not too hard calculation gives that
	\begin{equation}\label{eq:ProjSphere1}
		\rho(P_{(\cos\vartheta, e^{i\mu} \sin\vartheta)}) = \tfrac{1}{2} \cdot ( 1, \cos\mu \sin 2\vartheta, \sin\mu \sin 2\vartheta, \cos 2\vartheta ).
	\end{equation}
	Recall the remarkable angle doubling property of the Bloch representation, namely, we have $\|P-Q\| = \sin\theta$ if and only if the angle between the vectors $\rho(P) - \tfrac{1}{2}e_0$ and $\rho(Q) - \tfrac{1}{2}e_0$ is exactly $2\theta$.

	Next, we call a positive (semi-definite) element of $\Bdd_{sa}(\C^2)$ a density matrix if its trace is 1, or in other words, if it is a convex combination of some rank-one projections.
	Therefore $\rho$ maps the set of all $2\times 2$ density matrices onto the closed ball of the three-dimensional affine subspace $\{ (1/2,x_1,x_2,x_3) \colon x_j\in\R, j=1,2,3 \}$ with centre at $(1/2,0,0,0)$ and radius $1/2$.
	Hence, we see that the cone of all positive (semi-definite) $2\times 2$ matrices is mapped onto the infinite cone spanned by $(0,0,0,0)$ and the aforementioned ball. 
	Thus $\rho$ maps $\Effect(\C^2)$ onto the intersection of this cone and its reflection through the point $\rho(\tfrac{1}{2}I) = (\tfrac{1}{2},0,0,0)$.
	\begin{figure}[h]
		\begin{center}
			\includegraphics[scale=0.25]{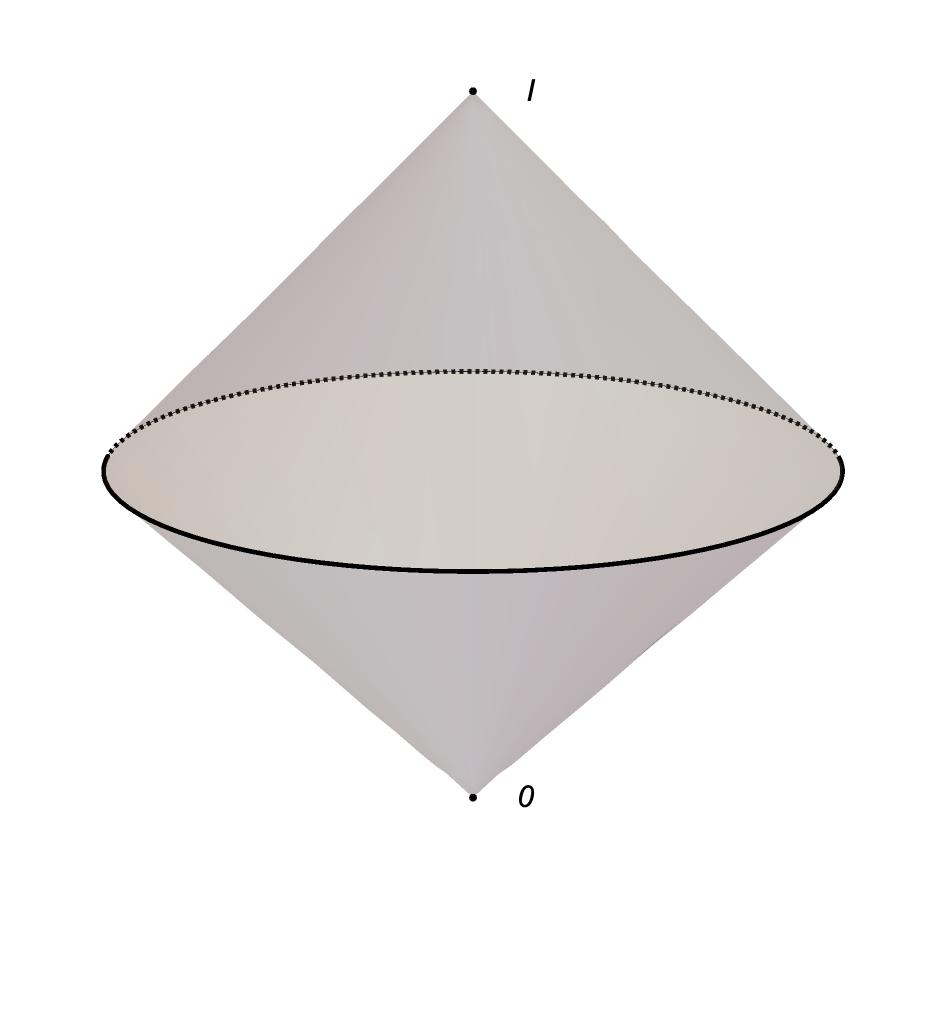}
			\caption{Illustration of $\rho(\Effect(\C^2))\cap S_\mu$. The circle is $\rho(\Proj_1(\C^2))\cap S_\mu$.}
			\label{fig:2}
		\end{center}
	\end{figure}

	We can re-write \eqref{eq:ProjSphere1} as follows:
	\begin{equation}\label{eq:ProjSphere}
		\rho(P_{(\cos\vartheta, e^{i\mu} \sin\vartheta)}) 
		= \tfrac{1}{2} \cdot ( e_0 + \sin 2\vartheta \cdot e_\mu + \cos 2\vartheta\cdot e_3),
	\end{equation}
	where
	$$
		e_0 := (1,0,0,0), \; e_\mu := (0, \cos\mu, \sin\mu, 0), \; e_3 := (0, 0, 0, 1)
	$$
	is an orthonormal system in $\mathbb{R}^4$.
	Let $S_\mu$ be the three-dimensional subspace spanned by $e_0, e_\mu, e_3$.
	Then the set $\rho(\Effect(\C^2))\cap S_\mu$ can be visualised as a double cone of $\R^3$, by regarding $e_0, e_\mu, e_3$ as the standard basis of $\R^3$, see Figure \ref{fig:2}.
	Note that $\rho(\Proj_1(\C^2))\cap S_\mu$ is the circle where the boundaries of the two cones meet.

	We continue with visualising the set $(tP_{(1, 0)})^\sim$ for an arbitrary $0<t<1$.
	Note that then visualising $(tP)^\sim$ for a general rank-one projection $P$ is very similar, we simply have to apply a unitary similarity (which by well-known properties of the Bloch representation, acts as a rotation on the sphere $\rho(\Proj_1(\C^2))$).
	Equation \eqref{eq:TA} gives the following:
	\begin{align}\label{eq:tP10F1Coex}
		&\Lambda \left( tP_{(1, 0)}, P_{(\cos\vartheta, e^{i\mu} \sin\vartheta)} \right) + \Lambda \left( I-tP_{(1,0)}, P_{(\cos\vartheta, e^{i\mu} \sin\vartheta)} \right) \nonumber \\
		&= \frac{1}{ \tfrac{1}{t} \cos^2\vartheta + \left(\tfrac{1}{0}\right) \sin^2\vartheta } + \frac{1}{ \tfrac{1}{1-t} \cos^2\vartheta + \sin^2\vartheta }
		= \left\{ \begin{matrix}
			\frac{1}{ \tfrac{1}{1-t} \cos^2\vartheta + \sin^2\vartheta } & \text{if}\; \vartheta > 0 \\
			1 & \text{if}\; \vartheta = 0 \\
		\end{matrix} \right..
	\end{align}	
	Now, let us consider the vector
	$$
		u = (2-t) \cdot e_0 + t \cdot e_3,
	$$
	which is orthogonal to 
	$$
		\rho\left((1-t)P_{(1,0)}-P_{(1,0)}^\perp\right) = -\tfrac{1}{2}\left[ t\cdot e_0 + (t-2)\cdot e_3\right].
	$$
	From here a bit tedious computation gives
	\begin{equation*}
		\left\langle u, \; \frac{1}{ \tfrac{1}{1-t} \cos^2\vartheta + \sin^2\vartheta } \cdot \rho\left( P_{(\cos\vartheta, e^{i\mu} \sin\vartheta)} \right) - \rho\left( P_{(1,0)}^\perp \right) \right\rangle = 0 \quad (0\leq \vartheta\leq \tfrac{\pi}{2}).
	\end{equation*}
	Therefore by Corollary \ref{cor:RankOneStrengthFunct} and \eqref{eq:tP10F1Coex} we conclude that $\rho\left((t P_{(1,0)})^\sim \cap \FiniteRank_1(\C^2)\right)$ is the union of the line segment $\{\rho\left(s P_{(1,0)}\right) \colon 0<s\leq 1 \} = \{ \tfrac{s}{2} e_0 + \tfrac{s}{2} e_3 \colon 0<s\leq 1 \}$ and of the area on the boundary of $\rho(\Effect(\C^2))$ which is either on, or below the affine hyperplane whose normal vector is $u$ and which contains $\rho(P_{(1,0)}^\perp)$, see Figure \ref{fig:3}.
	We note that using the notation of \eqref{eq:EtP}, the ellipse on the boundary is exactly the set 
	$$
	\left(\rho(\E_{tP_{(1,0)}})\cup \big\{(1-t)\cdot\rho(P_{(1,0)}) ,\rho(P_{(1,0)}^\perp)\big\}\right)\cap S_\mu.
	$$
	Therefore $\rho(\E_{tP_{(1,0)}})$ is a punctured ellipsoid.
	\begin{figure}[h]
		\begin{center}
			\includegraphics[scale=0.25]{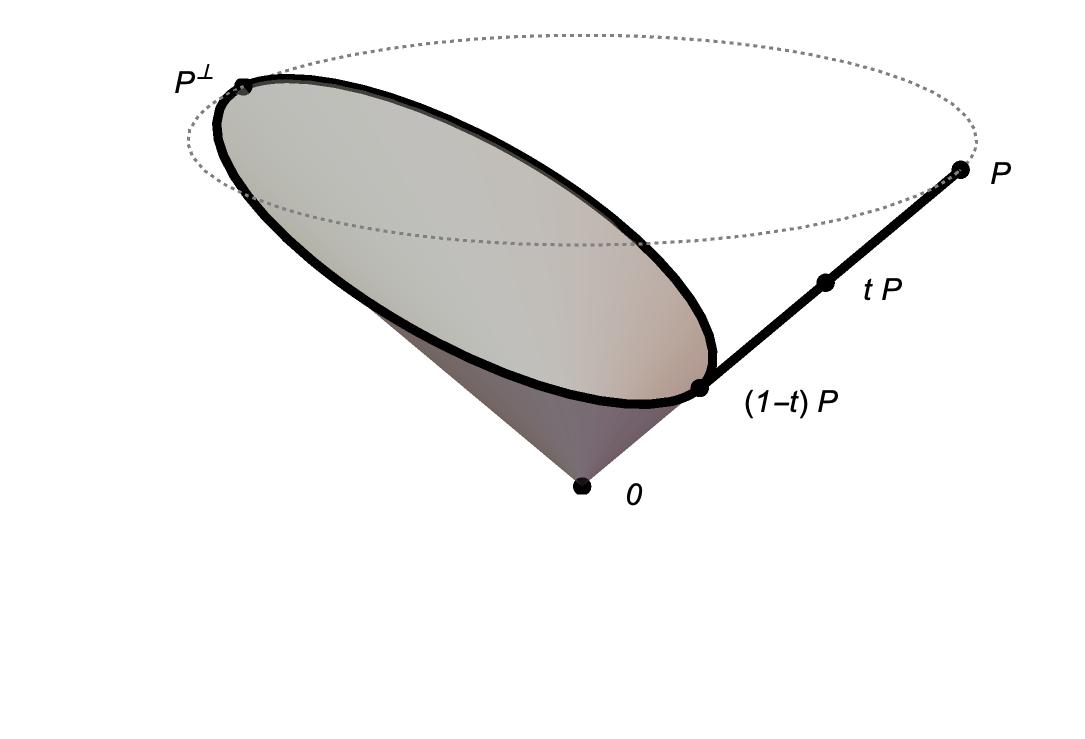}
			\caption{Illustration of $\rho\left((tP_{(1, 0)})^\sim\right)\cap\rho\left(\FiniteRank_1(\C^2)\right)\cap S_\mu$ (thick ellipse, thick line segment and the shaded area). The dotted circle is $\rho(\Proj_1(\C^2))\cap S_\mu$.}
			\label{fig:3}
		\end{center}
	\end{figure}

	If one illustrates the set $\rho\left((A)^\sim\right)\cap\rho\left(\FiniteRank_1(\C^2)\right)\cap S_\mu$ with $A, A^\perp \notin \Sca(\C^2) \cup \FiniteRank_1(\C^2)$ in the way as above, then one gets a set on the boundary of the cone which is bounded by a continuous closed curve containing the $\rho$-images of the spectral projections.

\section{Final remarks and open problems}

First, we prove the analogue of Lemma \ref{lem:2dprojchar} for finite dimensional spaces of dimension at least three.

\begin{lemma}\label{lem:E01minimal}
	Let $H$ be a Hilbert space with $2\leq \dim H < \infty$ and $A\in \Effect(H)$. 
	Then the following are equivalent:
	\begin{itemize}
		\item[\textup{(i)}] $0,1 \in \sigma(A)$,
		\item[\textup{(ii)}] there exists no effect $B\in \Effect(H)$ such that $B^\sim \subsetneq A^\sim$.
	\end{itemize}
\end{lemma}

\begin{proof}
	If $\dim H = 2$, then (i)$\iff$(ii) was proved in Lemma \ref{lem:2dprojchar}, so from now on we will assume $2 < \dim H < \infty$.
	Also, as the case when $A\in\Sca(H)$ is trivial, we assume otherwise throughout the proof.
	
	\emph{(i)$\Longrightarrow$(ii):}
	Suppose that $0,1 \in \sigma(A)$ and consider an arbitrary effect $B$ with $B^\sim \subseteq A^\sim$.
	By Lemma \ref{lem:CoexSubset}, $A$ and $B$ commute.
	If $0 = \lambda_1 \leq \lambda_2 \leq \dots \leq \lambda_{n-1} \leq \lambda_n = 1$ are the eigenvalues of $A$, then the matrices of $A$ and $B$ written in an orthonormal basis of joint eigenvectors are the following:
	$$
		A = \left[\begin{matrix}
				0 & 0 & \dots & 0 & 0 \\
				0 & \lambda_2 & \dots & 0 & 0 \\
				\vdots & & \ddots & & \vdots \\
				0 & 0 & \dots & \lambda_{n-1} & 0 \\
				0 & 0 & \dots & 0 & 1 \\
			\end{matrix}\right]
		\quad \text{and} \quad
		B = \left[\begin{matrix}
				\mu_1 & 0 & \dots & 0 & 0 \\
				0 & \mu_2 & \dots & 0 & 0 \\
				\vdots & & \ddots & & \vdots \\
				0 & 0 & \dots & \mu_{n-1} & 0 \\
				0 & 0 & \dots & 0 & \mu_n \\
			\end{matrix}\right]
	$$
	with some $\mu_1,\dots \mu_n \in [0,1]$.
	Notice that by Corollary \ref{cor:kulon}, for all $1\leq i < j \leq n$ we have
	\begin{equation}\label{eq:2dsiminlcusion}
		\left[\begin{matrix}
			\mu_i & 0 \\
			0 & \mu_j
		\end{matrix}\right]^\sim
		\subseteq
		\left[\begin{matrix}
			\lambda_i & 0 \\
			0 & \lambda_j
		\end{matrix}\right]^\sim.
	\end{equation}
	In particular, choosing $i = 1, j = n$ implies either $\mu_1 = 0$ and $\mu_n = 1$, or $\mu_1 = 1$ and $\mu_n = 0$.
	Assume the first case. 
	If we set $i=1$, then Lemma \ref{lem:F1order} and \eqref{eq:2dsiminlcusion} imply $\mu_j \geq \lambda_j$ for all $j = 2,\dots, n-1$.
	But on the other hand, setting $j=n$ implies $\mu_i \leq \lambda_i$ for all $i = 2,\dots, n-1$.
	Therefore we conclude $B = A$.
	Similarly, assuming the second case implies $B=A^\perp$.
	
	\emph{(ii)$\Longrightarrow$(i): }
	Assume (i) does not hold, then there exists a positive number $\varepsilon$ such that $\sigma (A) \subseteq [0,1-\varepsilon]$ or $\sigma (A^\perp) \subseteq [0,1-\varepsilon]$.
	Suppose the first possibility holds, then $\tfrac{1}{1-\varepsilon}A \notin \{A,A^\perp\}$ and $\left(\tfrac{1}{1-\varepsilon}A\right)^\sim \subseteq A^\sim$.
	The second case is very similar.
\end{proof}

We only proved the above lemma and Corollary \ref{cor:CoexSet2Dim} in the finite dimensional case. 
The following two questions would be interesting to examine:

\begin{question}
Does the statement of Corollary \ref{cor:CoexSet2Dim} remain true for general infinite dimensional Hilbert spaces?
\end{question}

\begin{question}
Does the statement of Lemma \ref{lem:E01minimal} hold if $\dim H \geq \aleph_0$?
\end{question}

Finally, our first main theorem characterises completely when $A^\sim = B^\sim$ happens for two effects $A$ and $B$. 
However, we gave only some partial results about when $A^\sim \subseteq B^\sim$ occurs, e.g.~Lemma \ref{lem:CoexSubset}.

\begin{question}
How can we characterise the relation $A^\sim \subseteq B^\sim$ for effects $A,B$?
\end{question}

We believe that a complete answer to this latter question would represent a substantial step towards the better understanding of coexistence.

\end{document}